\documentclass{amsart}
\usepackage{xspace, graphicx, epsfig, amssymb, amsmath, amsthm, stmaryrd}

\newcommand {\Z} {\mathbb{Z}}

\newcommand {\C} {\mathbb{C}}

\newcommand {\fl} {\mathrm{fl}}
\newcommand {\St} {\mathrm{St}}

\def\a{\mathfrak {a}}
\def\b{\mathfrak {b}}
\def\gl{\mathfrak {gl}}
\def\h{\mathfrak {h}}
\def\g{\mathfrak {g}}
\def\l{\mathfrak {l}}
\def\m{\mathfrak {m}}
\def\n{\mathfrak {n}}
\def\sl{\mathfrak{sl}}
\def\so{\mathfrak{so}}
\def\sp{\mathfrak{sp}}
\def\t{\mathfrak{t}}

\def\k{\mathfrak{k}}
\def\s{\mathfrak{s}}
\def\F{\mathfrak{F}}
\def\G{\mathfrak{G}}
\def\H{\mathfrak{H}}
\def\Ch{\mathcal{C}}
\def\Dh{\mathcal{D}}

\newcommand {\Span}{\mathrm{Span}}

\newcommand{\Sym}{\mathrm{Sym}}

\newcommand {\tw}{\mathrm{tw}}

\newtheorem{thm}{Theorem}[section]
\newtheorem{lemma}[thm]{Lemma}
\newtheorem{prop}[thm]{Proposition}

\newtheorem{defn}[thm]{Definition}

\begin{document}

\title{Borel subalgebras of root-reductive Lie algebras}
\thanks{Work supported by National Science Foundation grant DMS 0354321}
\author{Elizabeth Dan-Cohen}
\address{E. D.-C. Department of Mathematics \\ University of California at Berkeley \\ Berkeley, California 94720, USA}
\email{edc@math.berkeley.edu}

\date{July 18, 2007}

\begin{abstract}
This paper generalizes the classification in \cite{DP2} of Borel subalgebras of $\gl_\infty$.  Root-reductive Lie algebras are direct limits of finite-dimensional reductive Lie algebras along inclusions preserving the root spaces with respect to nested Cartan subalgebras.  A Borel subalgebra of a root-reductive Lie algebra is by definition a maximal locally solvable subalgebra.  The main general result of this paper is that a Borel subalgebra of an infinite-dimensional indecomposable root-reductive Lie algebra is the simultaneous stabilizer of a certain type of generalized flag in each of the standard representations.  

For the three infinite-dimensional simple root-reductive Lie algebras more precise results are obtained.  
The map sending a maximal closed (isotropic) generalized flag in the standard representation to its stabilizer hits Borel subalgebras, yielding a bijection in the cases of $\sl_\infty$ and $\sp_\infty$; in the case of $\so_\infty$ the fibers are of size one and two.  A description is given of a nice class of toral subalgebras contained in any Borel subalgebra.  Finally, certain Borel subalgebras of a general root-reductive Lie algebra are seen to correspond bijectively with Borel subalgebras of the commutator subalgebra, which are understood in terms of the special cases.
\end{abstract}

\maketitle

%%%%%%%%%%
%%%%%%%%%%
\section{Introduction}

The representation theory of root-reductive Lie algebras is currently being approached through a structure theory program.  Root-reductive Lie algebras are direct limits of finite-dimensional reductive Lie algebras along inclusions preserving the root spaces with respect to nested Cartan subalgebras.  The appropriate generalization in this context of the notion of a Borel subalgebra of a finite-dimensional Lie algebra is that of a \emph{maximal locally solvable} subalgebra.  This paper describes the Borel subalgebras of root-reductive Lie algebras, generalizing the results of \cite{DP2} in the case of $\gl_\infty$.  %This work is part of a structure theory program for root-reductive Lie algebras, which would lead to a study of weight representations.

The most general result of this paper, Theorem \ref{main}, states that a Borel subalgebra of an infinite-dimensional indecomposable root-reductive Lie algebra is the simultaneous stabilizer of a certain type of generalized flag in each of the standard representations.  Any root-reductive Lie algebra is the direct sum of finite-dimensional Lie algebras and infinite-dimensional indecomposable root-reductive Lie algebras.  Since Borel subalgebras of a direct sum are precisely direct sums of Borel subalgebras, the theorem can be used to understand Borel subalgebras of any root-reductive Lie algebra.  

Theorems \ref{slmain}, \ref{somain}, and \ref{spmain} address the infinite-dimensional simple root-reductive Lie algebras.  As in the case of $\gl_\infty$ treated in \cite{DP2}, Borel subalgebras of $\sl_\infty$ (or $\so_\infty$,  $\sp_\infty$) are stabilizers of maximal closed (isotropic) generalized flags in the standard representation.  The correspondence between Borel subalgebras and maximal closed (isotropic) generalized flags is bijective in the cases of $\gl_\infty$, $\sl_\infty$, and $\sp_\infty$; whereas a Borel subalgebra of $\so_\infty$ corresponds to one or two maximal closed isotropic generalized flags.  This phenomenon should not be surprising, since every Borel subalgebra of $\so_{2n}$ is the stabilizer of a pair of maximal isotropic flags in the standard representation.  We refer to any pair of maximal isotropic generalized flags corresponding to a single Borel subalgebra of $\so_\infty$ as \emph{twins}.

A nice class of toral subalgebras contained in a Borel subalgebra of $\sl_\infty$, $\so_\infty$, or $\sp_\infty$ is described in Section \ref{formulas}.  In these cases any Borel subalgebra is the span of such a toral subalgebra and the nilpotent subalgebra.  Thus irreducible representations of the Borel subalgebra are given by characters of the toral subalgebra.

Analysis of the general situation continues in Section \ref{general}.  In Theorem \ref{Borelcorrespondence} certain Borel subalgebras of a root-reductive Lie algebra $\g$ are seen to correspond bijectively to the Borel subalgebras of $[\g,\g]$.  It remains unknown whether every Borel subalgebra of $\g$ yields a Borel subalgebra of $[\g,\g]$ when intersected with $[\g,\g]$.

The argument which leads to the classification of Borel subalgebras of $\sl_\infty$, Theorem \ref{slmain}, begins with Theorem \ref{main} and Proposition \ref{special}, and continues with Lemmas \ref{slcyclic} and \ref{slinjective}.  Many elements of the proofs are straightforward applications to $\sl_\infty$ of work on $\gl_\infty$ seen in \cite{DP3}.  The proof of Theorem \ref{main}, by contrast, is quite different from their proof in the case of $\gl_\infty$; the modified proof allows for generalization to the isotropic cases.

I wish to acknowledge Ivan Dimitrov and Ivan Penkov for explaining their work in \cite{DP2}, and for sharing with me the proofs of the results announced there in the form of a manuscript \cite{DP3}.
The debt I owe Ivan Penkov goes further, for he introduced me to root-reductive Lie algebras and helped me greatly as I was first learning about them.  I wish to express my gratitude to Joseph Wolf, for his warm guidance and frequent attention throughout the writing of this paper.

%%%%%%%%%%
%%%%%%%%%%
\section{Preliminaries}

%%%%%%
%%%%%%
\subsection{Notation and a few definitions}

Let $V$ and $V_*$ be countable-dimensional vector spaces over the field of complex numbers $\C$, and let $\langle \cdot, \cdot \rangle : V \times V_* \rightarrow \C$ be a nondegenerate pairing.  We denote by $\gl(V,V_*)$ the Lie algebra associated to the associative algebra $V \otimes V_*$.  By $\sl(V,V_*)$  we denote the traceless part of $\gl(V,V_*)$, i.e. $[\gl(V,V_*) , \gl(V,V_*)]$.  If $\langle \cdot, \cdot \rangle : V \times V \rightarrow \C$ is a symmetric nondegenerate form, we denote by $\so(V)$ the Lie subalgebra $\bigwedge^2 V \subset \gl(V,V)$.  If $\langle \cdot, \cdot \rangle : V \times V \rightarrow \C$ is an antisymmetric nondegenerate form, we denote by $\sp(V)$ the Lie subalgebra $\Sym^2 (V) \subset \gl(V,V)$.  

By a result of Mackey  \cite[p. 171]{Mackey}, as long as the pairing $\langle \cdot , \cdot \rangle$ is nondegenerate, the above algebras do not depend on the pairing, up to isomorphism.  The usual representatives of these isomorphism classes are called $\gl_\infty$, $\sl_\infty$, $\so_\infty$, and $\sp_\infty$, respectively.

We will need a notion of the closure of a subspace of a vector space, with respect to a pairing.  That is, let $X$ and $Y$ be vector spaces, and let $\langle \cdot, \cdot \rangle : X \times Y \rightarrow \C$ be any pairing.  Given a subspace $F \subseteq X$, we consider the subspace $F^{\bot\bot}$, denoted $\overline{F}$, to be its \emph{closure} in $X$.  A subspace $F \subseteq X$ is said to be \emph{closed} if $F = \overline{F}$.  One important identity is that $F^\bot = F^{\bot \bot \bot}$ for any $F \subset X$.  As a result, for any $F \subset X$, the subspace $F^\bot \subset Y$ is closed.  Furthermore, the closure of any subspace is closed.  One may also check that the arbitrary intersection of closed subspaces is closed.  

If $F \subset X$ is a closed subspace and $F \subset G \subset X$ with $\dim G / F < \infty$, then $G$ is closed.  To see this, consider that $\dim G / F \leq \dim \overline{G} / F = \dim \overline{G} / \overline{F} \leq \dim F^\bot / G^\bot \leq \dim G / F$.  Hence $\dim \overline{G} / F = \dim G / F < \infty$, and since $G \subset \overline{G}$, we know $G = \overline{G} $.  

Now suppose $\langle \cdot, \cdot \rangle : V \times V \rightarrow \C$ is a nondegenerate pairing.  A subspace $F \subset V$ is said to be \emph{isotropic} if $F \subset F^\bot$, and \emph{coisotropic} if $F^\bot \subset F$.  If $F \subset V$ is an isotropic subspace, then its closure $\overline{F}$ is also isotropic.  That is, $F \subset F^\bot$ implies $\overline{F} \subset \overline{F^\bot}$, where $\overline{F^\bot} = F^{\bot \bot \bot} = \overline{F}^\bot$.

If $\langle \cdot , \cdot \rangle : V \times V \rightarrow \C$ is a symmetric nondegenerate form, an isotropic subspace $M \subset V$ is maximal isotropic if and only if $M$ is closed and $\dim M^\bot / M \leq 1$.  If $\langle \cdot , \cdot \rangle : V \times V \rightarrow \C$ is an antisymmetric nondegenerate form, a subspace $M \subset V$ is maximal isotropic if and only if $M = M^\bot $.

%%%%%%
%%%%%%
\subsection{Root-reductive Lie algebras}\label{standardrep}

A Lie algebra $\g$ is \emph{locally finite} if every finite subset of $\g$ is contained in a finite-dimensional subalgebra, i.e. if $\g$ is a union of finite-dimensional subalgebras.  One interesting class of locally finite Lie algebras is the root-reductive Lie algebras.

\begin{defn}
\begin{enumerate}
\item An inclusion of finite-dimensional reductive Lie algebras $\l \subseteq \m$ is called a \emph{root inclusion} if, for some Cartan subalgebra $\h_\m$ of $\m$, the subalgebra $\l \cap \h_\m$ is a Cartan subalgebra of $\l$ and each $\l \cap \h_\m$-root space $\l^\alpha$ is also a root space of $\m$.  
\item A Lie algebra $\g$ is called \emph{root-reductive} if it is isomorphic to a union $\bigcup_{i \in \Z_{>0}} \g_i$ of nested reductive Lie algebras with respect to root inclusions for a fixed choice of nested Cartan subalgebras $\h_i \subseteq \g_i$ with $\h_{i-1} = \h_i \cap \g_{i-1}$.
\end{enumerate}
\end{defn}

To understand the structure of root-reductive Lie algebras one uses the following theorem.

\begin{thm}  \label{structure} \cite{DP1}
Let $\g$ be a root-reductive Lie algebra.
\begin{enumerate}
\item \label{split} There is a split exact sequence of Lie algebras $$0 \rightarrow
[\g, \g] \rightarrow \g \rightarrow \g/[\g, \g] =: \a \rightarrow 0,$$
i.e. $\g \simeq [\g, \g] \subsetplus \a$, with $\a$ abelian.

\item The Lie algebra $[\g, \g]$ \label{simple} is isomorphic to a direct sum of finite-dimensional simple Lie algebras and copies of $\sl_\infty$, $\so_\infty$, and $\sp_\infty$.
\end{enumerate}
\end{thm}

Since there are no nontrivial extensions of of an abelian Lie algebra by a finite-dimensional simple Lie algebra, any root-reductive Lie algebra is isomorphic to a direct sum of finite-dimensional Lie algebras and a root-reductive Lie algebra $\g$ in which $[\g,\g]$ is isomorphic to a direct sum of copies of $\sl_\infty$, $\so_\infty$, and $\sp_\infty$.

Let $\g$ be an infinite-dimensional indecomposable root-reductive Lie algebra.  Then $[\g,\g] \cong \bigoplus_m \s_m$ as Lie algebras, where for each $m$ the component $\s_m$ is isomorphic to $\sl_\infty$, $\so_\infty$, or $\sp_\infty$.   Let $V_m$ denote the standard representation of $\s_m$, and let $(V_m)_*$ denote the relevant dual representation.  Consider $V_m$ as a $[\g,\g]$-module on which $\bigoplus_{n \neq m} \s_n$ acts trivially.  By Proposition 4.2 of \cite{DPS}, there exists a $\g$-module structure on $V_m$ extending the $[\g,\g]$-module structure.  Likewise, there exists a $\g$-module structure on $(V_m)_*$ in which $\bigoplus_{n \neq m} \s_n$ acts trivially.  One may check that under this construction, the pairing $\langle \cdot , \cdot \rangle : V_m \times (V_m)_* \rightarrow \C$ is $\g$-invariant.  By \emph{the standard representations} of $\g$, we mean the representations $V_m$ together with a choice of $\g$-module structure on each.

A subalgebra of a root-reductive Lie algebra is called a \emph{toral subalgebra} if it consists of elements which are semisimple in the sense of Jordan decomposition \cite{DPS}.  We denote the normalizer in $\g$ of a subalgebra $\k$ by $\n_\g(\k)$. 
 
A locally finite Lie algebra $\s$ is \emph{locally solvable} if every finite subset of $\s$ is contained in a solvable subalgebra, i.e. if $\s$ is a union of finite-dimensional solvable subalgebras.  The \emph{nilpotent subalgebra} of a locally solvable Lie algebra is defined to be the span of all elements which are nilpotent in the sense of Jordan decomposition.  The following is an essential result about representations of locally solvable Lie algebras:
\begin{lemma}\cite{DP3} \label{irreducible}
Let $\s$ be a locally finite locally solvable Lie algebra, i.e. $\s = \bigcup_i \s_i$ with $\s_i$ finite-dimensional and solvable.  If $W$ is an irreducible $\s$-module which is a union of finite-dimensional $\s_i$-modules $W_i$, then $\dim W = 1$.
\end{lemma}
Finally, a subalgebra of a root-reductive Lie algebra is called a \emph{Borel subalgebra} if it is maximal locally solvable.

%%%%%%
%%%%%%
%\subsection{Standard representations of a root-reductive Lie algebra}\label{standardrep}

%Let $\g$ be an infinite-dimensional indecomposable root-reductive Lie algebra.  Recall that $[\g,\g] \cong \bigoplus_m \s_m$ as Lie algebras, where for each $m$ the component $\s_m$ is either a finite-dimensional simple Lie algebra or isomorphic to one of $\sl_\infty$, $\so_\infty$, and $\sp_\infty$.  The assumption that $\g$ is indecomposable and not finite-dimensional implies that $\s_m$ is infinite-dimensional for all $m$.  Let $V_m$ denote the standard representation of $\s_m$, and let $(V_m)_*$ denote the relevant dual representation.  Consider $V_m$ as a $[\g,\g]$-module on which $\bigoplus_{n \neq m} \s_n$ acts trivially.  By Proposition 4.2 of \cite{DPS}, there exists a $\g$-module structure on $V_m$ extending the $[\g,\g]$-module structure.  Likewise, there exists a $\g$-module structure on $(V_m)_*$ in which $\bigoplus_{n \neq m} \s_n$ acts trivially.  One may check that under this construction, the pairing $\langle \cdot , \cdot \rangle : V_m \times (V_m)_* \rightarrow \C$ is $\g$-invariant.  By \emph{the standard representations} of $\g$, we mean the representations $V_m$ together with a choice of $\g$-module structure on each.

%%%%%%
%%%%%%
\subsection{Generalized flags}\label{generalizedflagterminology}

The definitions in this section, apart from those of bivalent closed and Borel generalized flag, were made in \cite{DP2} to study Borel subalgebras of $\gl_\infty$.  Let $X$ be a complex vector space.  A \emph{chain} in $X$ is a set of subspaces of $X$ totally ordered by inclusion.  A \emph{generalized flag $\F$ in $X$} is a chain in $X$ such that each subspace $F \in \F$ has an immediate predecessor or an immediate successor in the inclusion ordering, and for every nonzero $x \in X$ there exists an immediate predecessor-successor pair $F' \subset F'' \in \F$ with $x \in F'' \setminus F'$.  Let $A$ be the ordered set of immediate predecessor-successor pairs, and denote by $F'_\alpha$ the predecessor and by $F''_\alpha$ the successor of each pair $\alpha \in A$.  Since every subspace in $\F$ is either the immediate predecessor or the immediate successor of another subspace, a generalized flag $\F$ may be considered as $\F = \{ F'_\alpha , F''_\alpha \}_{\alpha \in A}$.

Let $x \in X$ be nonzero.  Then we denote by $F'_x$ and $F''_x$ the predecessor and successor, respectively, of the immediate predecessor-successor pair such that $x \in F''_x \setminus F'_x$, obtained from the definition of a generalized flag.  A generalized flag $\G$ is considered to be a \emph{refinement} of $\F$ if $F'_x \subset G'_x \subset G''_x \subset F''_x$ for every nonzero $x \in X$.  A generalized flag $\F = \{ F'_\alpha , F''_\alpha \}_{\alpha \in A}$ is maximal (with respect to refinements) if $\dim F''_\alpha / F'_\alpha = 1$ for all $\alpha \in A$.

Suppose $\Ch$ is a chain of subspaces in $X$ satisfying the property that for each $x \in X$, there exists a subspace $C \in \Ch$ containing $x$, as well as a subspace $C \in \Ch$ not containing $x$.  (This is not terribly restrictive, as one sufficient condition is that $0$ and $X$ be elements of $\Ch$.)  Then one may obtain a generalized flag $\fl(\Ch)$ by defining: $$\fl(\Ch):= \{ \bigcup_{x \notin C \in \Ch} C , \bigcap_{x \in C \in \Ch} C \}_{ 0 \neq x \in X }.$$
If $\F = \fl (\Ch)$, then for each nonzero $x \in X$, one has $F'_x = \bigcup_{x \notin C \in \Ch} C $ and $F''_x = \bigcap_{x \in C \in \Ch} C$.  The generalized flag obtained from a chain is not necessarily a subset of that chain, nor must it contain every subspace in the chain.  Take as an example a chain of the form $$0 \subsetneq V_1 \subsetneq V_2 \subsetneq V_3 \subsetneq \cdots \subsetneq \bigcup V_i \subsetneq \cdots \subsetneq W_3 \subsetneq W_2 \subsetneq W_1 \subsetneq X.$$  If on the one hand $\bigcup_i V_i \neq \bigcap_j W_j$, then the generalized flag obtained from this chain is $$0 \subsetneq V_1 \subsetneq V_2 \subsetneq V_3 \subsetneq \cdots \subsetneq \bigcup V_i \subsetneq \bigcap W_j \subsetneq \cdots \subsetneq W_3 \subsetneq W_2 \subsetneq W_1 \subsetneq X.$$  If on the other hand $\bigcup_i V_i = \bigcap_j W_j$, then the generalized flag obtained from this chain is $$0 \subsetneq V_1 \subsetneq V_2 \subsetneq V_3 \subsetneq \cdots \subsetneq \cdots \subsetneq W_3 \subsetneq W_2 \subsetneq W_1 \subsetneq X.$$  

Now suppose that there is a bilinear form $X \times Y \rightarrow \C$.  For any chain $\Ch$ of subspaces of $X$, one may consider the set of subspaces given by $\Ch^\bot := \{ C^\bot \}_{C \in \Ch}$, which is a chain in $Y$.  A generalized flag $\F$ is said to be \emph{closed} if $\F = \fl( \F^{\bot\bot})$.   A generalized flag is closed if and only if every immediate succesor is closed while every immediate predecessor has as its closure either itself or its immediate successor.  In the context of closed generalized flags, we use the term \emph{good pair} to refer to any immediate predecessor-successor pair of which the predecessor is closed.  A closed generalized flag is a maximal closed generalized flag if and only if every good pair has codimension $1$.

We say that a closed generalized flag $\F$ is \emph{bivalent} if every good pair has codimension $1$ or $\infty$.  Let $\F$ be a bivalent closed generalized flag in $X$.  A generalized flag $\G$ refining $\F$ is called \emph{Borel}\footnote{It may turn out that the only Borel generalized flags which are of interest are the maximal closed generalized flags.} if whenever a nonzero $x \in X$ yields a good pair with infinite codimension $F'_x \subset F''_x$ in $\F$, it holds that $\dim G''_x / G'_x = 1$ and $\overline{G'_x} = F''_x$; and otherwise $F'_x = G'_x \subset G''_x = F''_x$.  Note that maximal closed generalized flags are a subset of the bivalent closed generalized flags, and that any maximal closed generalized flag may be considered as a Borel generalized flag refining itself.
 
If $\F = \{F'_\alpha , F''_\alpha \}$ is a generalized flag in $V$, then the stabilizer of $\F$ in $\gl(V,V_*)$ may be calculated as $\St_\F = \sum_\alpha F''_\alpha \otimes (F'_\alpha)^\bot$ \cite{DP2}.  This is not hard to check, and a proof is given in \cite{DP3}.  Also, the span of the nilpotent elements of $\St_\F$ (that is to say its nilpotent subalgebra, since $\St_\F$ is locally solvable as seen below) is given by the formula $\sum_\alpha F''_\alpha \otimes (F''_\alpha)^\bot$ \cite{DP2}.

The following proposition is a consequence of a more complicated statement in \cite{DP3}, and I present an alternative proof.

\begin{prop} \label{locsolvable}
Let $\F$ be a maximal generalized flag in $V$.  Then the stabilizer in $\gl(V,V_*)$ of $\F$ is a locally solvable subalgebra.
\end{prop}

\begin{proof}
Let $X \subset V$ and $Y \subset V_*$ be finite-dimensional subspaces such that the restriction of $\langle \cdot , \cdot \rangle$ to $X \times Y$ is nondegenerate.  Let $d$ be the dimension of $X$, and of course $X \otimes Y \cong \gl_d$.  Observe that $\gl(V,V_*)$ may be exhausted by such subalgebras.

For $i = 1 , \ldots , d$ there exists $\alpha_i \in A$ such that $\dim (X \cap F''_{\alpha_i} )= i$.  Let $X_i := X \cap F''_{\alpha_i}$.  Then $$ 0 \subset X_1 \subset \cdots \subset X_{d-1} \subset X_d = X$$ is a maximal flag in $X$.  Choose for each $i$ an element $x_i \in X_i \setminus X_{i-1}$.  For each $i$ there exists $\beta_i \in A$ such that $x_i \in F''_{\beta_i} \setminus F'_{\beta_i}$.  Then $\St_\F \cap (X \otimes Y) = \sum_{i=1}^d X_i \otimes ((F'_{\beta_i})^\bot \cap Y)$.

One may check that $(F'_{\beta_i})^\bot \cap Y \subset X_{i-1}^\bot$, where the perpendicular complement of $F'_{\beta_i}$ is taken in $V_*$ and the perpendicular complement of $X_{i-1}$ is taken in $Y$.  This follows immediately from the fact that $X_{i-1} \subset F'_{\beta_i}$.  Therefore $\St_\F \cap (X \otimes Y)  \subset  \sum_{i=1}^d X_i \otimes X_{i-1}^\bot$.  The latter expression is the stabilizer of the maximal flag $X$ in $X \otimes Y$, which is solvable since it is a Borel subalgebra.  Therefore $\St_\F \cap (X \otimes Y)$ is solvable.  It follows that $\St_\F$ is locally solvable.
\end{proof}

%%%%%%
%%%%%%
\subsection{Isotropic generalized flags}

Now let $\langle \cdot , \cdot \rangle : V \times V \rightarrow \C$ be a bilinear form.  We will say that $\F$ is an \emph{isotropic generalized flag in $V$} if every $F \in \F$ is an isotropic subspace of $V$, and $\F$ is a generalized flag in $\bigcup_{F \in \F} F$.  As before, we say an isotropic generalized flag $\F$ is \emph{closed} if $\F = \fl( \F^{\bot\bot})$.  Again, an isotropic generalized flag is closed if and only if every immediate succesor is closed while every immediate predecessor has as its closure either itself or its immediate successor.  A closed isotropic generalized flag $\F$ in $V$ is a maximal closed isotropic generalized flag if and only if the subspace $\bigcup_{F \in \F} F$ is a maximal isotropic subspace of $V$, and every good pair has codimension $1$.

If $\F$ is a generalized flag in $V$, let $\F_{iso}$ denote the pairs of $\F$ which are isotropic, i.e. $\F_{iso} := \{ F'_\alpha , F''_\alpha : F''_\alpha \subset (F''_\alpha)^\bot \}$.

%%%%%%%%%%
%%%%%%%%%%
\section{Subspaces stable under a Borel subalgebra}\label{stablesubspaces}

%%%%%%
%%%%%%
\subsection{Generalized flags in the standard representations}

The following result motivates the definition given in section \ref{generalizedflagterminology} of a Borel generalized flag.

\begin{thm}\label{main}
Any Borel subalgebra of an infinite-dimensional indecomposable root-reductive Lie algebra is the simultaneous stabilizer of a Borel generalized flag in each of the standard representations.
\end{thm}

\begin{proof}
Let $\g$ be an infinite-dimensional indecomposable root-reductive Lie algebra, and let $\b \subset \g$ be a Borel subalgebra.  Let $[\g,\g] \cong \bigoplus_m \s_m$ be the decomposition into simple root-reductive Lie algebras, where $\s_m$ is isomorphic to one of $\sl_\infty$, $\so_\infty$, and $\sp_\infty$ for each $m$.  Let $V_m$ denote the standard representations of $\g$, as defined in section \ref{standardrep}.  For each $m$, let $\Ch_m$ be a maximal chain of closed $\b$-stable subspaces in $V_m$.  Take $\F_m := \fl(\Ch_m)$.

Let $F' \subset F''$ be any immediate predeccessor-successor pair in $\F_m$.  One can see immediately that there are no closed subspaces properly between $F'$ and $F''$.  Observe that $F''$ is closed, since it is obtained as the intersection of closed subspaces of $V_m$.  If $F'$ is not closed, then $\overline{F'} = F''$ because there are no closed subspaces properly between $F'$ and $F''$.  This implies that $\F_m$ is a closed generalized flag.

If $F'$ is closed, then $\dim(F'' / F')$ is either $1$ or infinite.  In detail, let $G$ be any $\b$-stable subspace $F' \subset G \subset F''$.  If $\dim F'' / F' < \infty$, then $\dim G / F' < \infty$, and hence $G$ is closed, which implies that $G$ is equal to either $F'$ or $F''$.  That is, if $\dim F'' / F' < \infty$, then $F'' / F'$ is an irreducible $\b$-module, hence $1$-dimensional.  Thus $\F_m$ is a bivalent closed generalized flag.

Let $\Dh_m$ be obtained from $\F_m$ by adding a maximal chain of $\b$-stable subspaces between every pair good $F' \subset F''$ with $\dim(F''/F') = \infty$.  Let $\G_m := \fl(\Dh_m)$.  Clearly $\G_m$ is a refinement of $\F_m$.  Let $0 \neq x \in V_m$ be such that $F'_x$ closed and $\dim F''_x / F'_x = \infty$.  The maximality of $\Dh_m$ implies $\dim G''_x / G'_x = 1$.  Moreover, $F'_x \subsetneq G'_x$, otherwise $G''_x$ would be a closed $\b$-stable subspace.  Similarly, $\overline{G'_x}$ is a closed $\b$-stable subspace with $F'_x \subsetneq \overline{G'_x} \subset F''_x$, and therefore $\overline{G'_x} = F''_x$.  Thus $\G_m$ is a Borel generalized flag refining $\F_m$.

Consider $\s_m \subset \gl(V_m,(V_m)_*)$.  Observe that the stabilizer in $\gl(V_m,(V_m)_*)$ of $\G_m$ is equal to the stabilizer in $\gl(V_m,(V_m)_*)$ of any maximal generalized flag refining $\G_m$, which by Proposition \ref{locsolvable} is locally solvable.  As a result, $\St_{\G_m} \cap \s_m$ is locally solvable.

Since each flag $\G_m$ is stable under $\b$, indeed $\b \subset \bigcap_m \St_{\G_m}$.
Calculate
\begin{eqnarray*}
[ \bigcap_m \St_{\G_m} ,  \bigcap_m \St_{\G_m}] & \subset & ( \bigcap_m \St_{\G_m}) \cap [\g,\g] \\
& = & \bigoplus_m (\St_{\G_m} \cap \s_m).
\end{eqnarray*}
Since each $\St_{\G_m} \cap \s_m$ is locally solvable, it follows that $\bigoplus_m (\St_{\G_m} \cap \s_m)$ is locally solvable.  Therefore $ \bigcap_m \St_{\G_m}$ is a locally solvable subalgebra of $\g$.  Since $\b$ is maximal locally solvable, finally $\b = \bigcap_m \St_{\G_m}$.
\end{proof}

The general case is resumed in Section \ref{general}.

%%%%%%
%%%%%%
\subsection{Isotropic subspaces in the standard representation of $\so_\infty$}

Suppose $\b \subset \so(V)$ is a Borel subalgebra.

\begin{prop} \label{somaximalisotropic}
A maximal $\b$-stable isotropic subspace of $V$ is maximal isotropic.
\end{prop}

\begin{lemma} \label{sointersection}
Suppose $M \subset V$ is a maximal $\b$-stable isotropic subspace.  If $G$ is a $\b$-stable subspace with $M \subset G \subset M^\bot$, then $G \cap G^\bot = M$.
\end{lemma}

\begin{proof}[Proof of Lemma \ref{sointersection}.]
Observe that $M \subset G^\bot \subset M^\bot$.  Since $G \cap G^\bot$ is $\b$-stable and isotropic, and moreover $M \subset G \cap G^\bot$, the maximality of $M$ implies $G \cap G^\bot =M$.
\end{proof}

\begin{proof}[Proof of Proposition \ref{somaximalisotropic}.]
Since $M$ is isotropic, its closure $\overline{M}$ is also isotropic.  Moreover, $\overline{M}$ is stable under $\b$ because $M$ is stable under $\b$, by the $\g$-invariance of $\langle \cdot, \cdot \rangle$.  By the maximality of $M$, indeed $M$ is closed.

Let $\Ch$ be a maximal chain of $\b$-stable subspaces of $V$ between $M$ and $M^\bot$.  Let $\F := \fl (\Ch)$, so that $0 \subset M \subset \F \subset M^\bot \subset V$ is a generalized flag in $V$.  Write $\F = \{ F'_\alpha , F''_\alpha \}_{\alpha \in A}$.  By Lemma \ref{irreducible}, irreducible $\b$-modules are one dimensional, so it must be that $\dim F''_\alpha / F'_\alpha = 1$ for all $\alpha \in A$.

Suppose $Y \in \b$.  Since $Y$ stabilizes the generalized flag $0 \subset M \subset \F \subset M^\bot \subset V$, it follows that $Y \in M \otimes V + \sum_\alpha F''_\alpha \otimes (F'_\alpha)^\bot + V \otimes M$.  I will show that in fact $Y \in M \otimes V + V \otimes M$.

Now $Y =  \sum_i v_i \otimes w_i + Z$, for some $v_i \in F''_{\alpha_i} \setminus F'_{\alpha_i}$ and $w_i \in (F'_{\alpha_i})^\bot \setminus M$ and $Z \in M \otimes V +  V \otimes M$.  One may safely assume that the set $\{ v_i \}$ is linearly independent modulo $M$ and modulo $F'_\beta$ for all $\beta$.

Let $\sigma : V \otimes V \rightarrow V \otimes V$ denote the linear map which swaps the two factors.  Since $Y \in \bigwedge^2 V$, we calculate $-Y = \sigma(Y) =  \sum_i w_i \otimes v_i + \sigma(Z)$.  Hence
\begin{equation} \label{somadness}
\sum_i v_i \otimes w_i + w_i \otimes v_i = -Z - \sigma(Z).
\end{equation}
Looking at the left hand side of (\ref{somadness}), one can see that $\sum_i v_i \otimes w_i + w_i \otimes v_i$  is an element of $M^\bot \otimes M^\bot$, whereas the right hand side is an element of $M \otimes V + V \otimes M$.  Hence the right hand side of (\ref{somadness}) is an element of $(M^\bot \otimes M^\bot) \cap (M \otimes V + V \otimes M) = M \otimes M^\bot + M^\bot \otimes M$.

For each $i$ we have $w_i \in M^\bot \setminus M$, so there exists $\beta_i \in A$ such that $w_i \in F''_{\beta_i} \setminus F'_{\beta_i}$.  Since $w_i \in F''_{\beta_i} \cap (F'_{\alpha_i})^\bot$, Lemma \ref{sointersection} implies that $\beta_i \geq \alpha_i$.

Assume, for the sake of a contradiction, that $\sum_i v_i \otimes w_i$ is nonzero.  Let $\beta := 
\mathrm{max}_i \{ \beta_i \}$, where this set is nonempty by hypothesis.  So $\beta = \beta_1 = \cdots = \beta_k$ and $\beta > \beta_i$ for $i \neq 1 , \ldots , k$.  Meanwhile $\beta \geq \alpha_i$ for all $i$.  By assumption $\{ v_i  \}$ is linearly independent modulo $F'_\beta$, so in fact $\alpha_i = \beta$ for at most one $i$.

\begin{enumerate}
\item \label{one} First suppose that $\alpha_1 = \beta$.  Then $\alpha_i < \beta$ for $i \neq 1$, i.e. $v_i \in F'_\beta$ for $i \neq 1$.  Equation (\ref{somadness}) yields
\begin{eqnarray*}
v_1 \otimes w_1+ w_1 \otimes v_1 & = & - \sum_{i \neq 1} (v_i \otimes w_i + w_i \otimes v_i) - Z - \sigma(Z)\\
& \in & F'_\beta \otimes M^\bot + M^\bot \otimes F'_\beta. 
\end{eqnarray*}
This contradicts the fact that  $v_1, w_1 \in F''_\beta \setminus F'_\beta$.

\item Now suppose that $\alpha_i < \beta$ for all $i$.  For $i = 1 , \ldots , k$, there exist unique $b_i \in \C$ and $w_i' \in F'_\beta$ such that $w_i = b_i w_1 + w_i'$.  Then equation (\ref{somadness}) yields
\begin{eqnarray*}
w_1 \otimes (b_1 v_1 + \cdots + b_k v_k) & = & - \sum_{i=1}^k w_i' \otimes v_i - \sum_{i \neq 1 , \ldots , k} w_i \otimes v_i - \sum_{i} v_i \otimes w_i \\
&&- Z - \sigma(Z) \\
& \in & F'_\beta \otimes M^\bot + M^\bot \otimes M.
\end{eqnarray*}
Since $w_1 \notin F'_\beta$, it follows that $b_1 v_1 + \cdots + b_k v_k \in M$.  The fact that $b_1 = 1$ contradicts the assumption that the set $\{ v_i \}$ is linearly independent modulo $M$. 
\end{enumerate}

Either case leads to a contradiction.  Therefore $\sum_i v_i \otimes w_i = 0$, and $Y = Z \in M \otimes V + V \otimes M$.

Thus $Y \cdot M^\bot \subset M$.  Since $Y \in \b$ was arbitrary, indeed $\b \cdot M^\bot \subset M$.  Let $L$ be any isotropic subspace containing $M$.  Then $M \subset L \subset M^\bot$, so $L$ is stable under $\b$.  Since $M$ is a maximal $\b$-stable isotropic subspace, $L = M$.  Therefore $M$ is a maximal isotropic subspace.
\end{proof}

\begin{prop} \label{sochain}
There exists a maximal isotropic subspace $M \subset V$ which is stable under $\b$.  Furthermore, there exists a maximal chain $\Ch$ of closed $\b$-stable subspaces in $V$ containing $M$, with the additional property that $\Ch^\bot \subset \Ch$.  
\end{prop}

\begin{proof}
As a corollary to Proposition \ref{somaximalisotropic}, there exists a maximal isotropic subspace $M \subset V$ which is stable under $\b$.   (Observe that $0$ is a $\b$-stable isotropic subspace of $V$, and that the union of nested $\b$-stable isotropic subspaces is a $\b$-stable isotropic subspace.  Hence there exists a subspace $M \subset V$ which is a maximal $\b$-stable isotropic subspace.  By Proposition \ref{somaximalisotropic}, $M$ is a maximal isotropic subspace of $V$.)  

Suppose $\Ch$ is a chain of closed $\b$-stable subspaces with $M \in \Ch$ and $\Ch^\bot \subset \Ch$.  Suppose further that $D$ is a closed $\b$-stable subspace such that $\Ch \cup \{ D \}$ is a chain.  Then $\Dh := \Ch \cup \{ D , D^\bot \}$ is a chain of closed $\b$-stable subspaces with $M \in \Dh$ and such that $\Dh^\bot \subset \Dh$.  To see that $\Dh$ is a chain, consider first the fact that since $M$ and $M^\bot$ are elements of $\Ch$, the subspace $D$ is either isotropic or coisotropic, i.e. either $D \subset D^\bot$ or $D^\bot \subset D$.  It remains to show that for any $C \in \Ch$, either $C \subset D^\bot$ or $D^\bot \subset C$.  But this follows immediately from the fact that either $D \subset C^\bot$ or $C^\bot \subset D$, since $C^\bot \in \Ch$ and $C$ is closed.  Hence a chain which is maximal with respect to chains $\Ch$ of closed $\b$-stable subspaces containing $M$ such that $\Ch^\bot \subset \Ch$ is in fact a maximal chain of closed $\b$-stable subspaces.
\end{proof}

%%%%%%
%%%%%%
\subsection{Isotropic subspaces in the standard representation of $\sp_\infty$}

Suppose $\b \subset \sp(V)$ is a Borel subalgebra.  The propositions in this section are completely analogous to those in the previous section, but their proofs admit significant simplifications.

\begin{prop} \label{spmaximalisotropic}
A maximal $\b$-stable isotropic subspace of $V$ is maximal isotropic.
\end{prop}

\begin{lemma} \label{spintersection}
Suppose $M \subset V$ is a maximal $\b$-stable isotropic subspace.  If $G'$ and $G''$ are $\s$-stable subspaces with $M \subset G' \subset G'' \subset M^\bot$ and $\dim G'' / G' = 1$, then $G'' \cap (G')^\bot = M$.
\end{lemma}

\begin{proof}[Proof of Lemma \ref{spintersection}.]
Observe that $M \subset (G')^\bot \subset M^\bot$.  Since $G' \cap (G')^\bot$ is $\b$-stable and isotropic, and moreover it contains $M$, the maximality of $M$ implies $G' \cap (G')^\bot =M$.  The inclusion $M = G' \cap (G')^\bot \subset G'' \cap (G')^\bot$ has codimension $0$ or $1$.  Suppose, for the sake of a contradiction, that $G'' \cap (G')^\bot = M \oplus \C x$.  Then $x \in M^\bot$, and $\langle x , x \rangle = 0$ since the pairing $\langle \cdot , \cdot \rangle$ is antisymmetric.  Hence $\langle M \oplus \C x , M \oplus \C x \rangle = 0$, and $M \oplus \C x$ is isotropic.  It is also $\b$-stable.  This contradicts the maximality of $M$.  Hence $G'' \cap (G')^\bot = M$.
\end{proof}

\begin{proof}[Proof of Proposition \ref{spmaximalisotropic}.]
Let $\F = \{F'_\alpha, F''_\alpha \}_\alpha$ be defined in the same fashion as in the proof of Proposition \ref{somaximalisotropic}.  Suppose $Y \in \b$.  Again, $Y \in M \otimes V + \sum_\alpha F''_\alpha \otimes (F'_\alpha)^\bot + V \otimes M$, so $Y =  \sum_i v_i \otimes w_i + Z$, for some $v_i \in F''_{\alpha_i} \setminus F'_{\alpha_i}$ and $w_i \in (F'_{\alpha_i})^\bot \setminus M$ and $Z \in M \otimes V +  V \otimes M$.  One may safely assume that the set $\{ v_i \}$ is linearly independent modulo $M$ and modulo $F'_\beta$ for all $\beta$.
  
Since $Y \in \Sym^2(V)$, we calculate $Y = \sigma(Y) =  \sum_i w_i \otimes v_i + \sigma(Z)$.  Hence
\begin{equation} \label{spmadness}
\sum_i v_i \otimes w_i - w_i \otimes v_i = Z - \sigma(Z).
\end{equation}
As in the proof of Proposition \ref{somaximalisotropic}, the right hand side of (\ref{spmadness}) is an element of $M \otimes M^\bot + M^\bot \otimes M$.

For each $i$ we have $w_i \in F''_{\beta_i} \setminus F'_{\beta_i}$, where $M \subset F'_{\beta_i} \subset F''_{\beta_i} \subset M^\bot$.  Since $w_i \in F''_{\beta_i} \cap (F'_{\alpha_i})^\bot$, we obtain from Lemma \ref{spintersection} that $\beta_i > \alpha_i$.  The rest of the proof follows the same outline as the proof of Proposition \ref{somaximalisotropic}, with the simplification that Case (\ref{one}) has already been ruled out.
\end{proof}

\begin{prop} \label{spchain}
There exists a maximal isotropic subspace $M \subset V$ which is stable under $\b$.  Furthermore, there exists a maximal chain $\Ch$ of closed $\b$-stable subspaces in $V$ containing $M$, with the additional property that $\Ch^\bot \subset \Ch$.
\end{prop}

The proof is identical to that of Proposition \ref{sochain}.

%%%%%%
%%%%%%
\subsection{Maximal closed generalized flags in the standard representation}

The following proposition is an improvement of Theorem \ref{main} in the special cases of the infinite-dimensional simple root-reductive Lie algebras.

\begin{prop} \label{special}
Any Borel subalgebra of $\sl(V,V_*)$ is the stabilizer of a maximal closed generalized flag in $V$.  Any Borel subalgebra of $\so(V)$ or $\sp(V)$ is the stabilizer of a maximal closed generalized flag $\F$ in $V$ with $\F \cup \F^\bot \cup \{M , M^\bot \}$ a chain for some maximal isotropic subspace $M \subset V$.
\end{prop}

\begin{proof}
If $\g = \sl(V,V_*)$, let $\Ch$ be a maximal chain of closed $\b$-stable subspaces in $V$.  If $\g = \so(V)$, let $M$ be a $\b$-stable maximal isotropic subspace in $V$, and let $\Ch$ be a maximal chain of closed $\b$-stable subspaces in $V$, with $M \in \Ch$ and $\Ch^\bot \subset \Ch$, as in Proposition \ref{sochain}.  If $\g = \sp(V)$,  let $M$ be a $\b$-stable maximal isotropic subspace in $V$, and let $\Ch$ be a maximal chain of closed $\b$-stable subspaces in $V$, with $M \in \Ch$ and $\Ch^\bot \subset \Ch$, as in Proposition \ref{spchain}.

Let $\F : = \fl(\Ch)$, as in the proof of Theorem \ref{main}.  Observe that if $\g$ is one of $\so(V)$ and $\sp(V)$, then $\F \cup \F^\bot \cup \{M , M^\bot \}$ is a chain.  That is, the maximality of $\Ch$ implies that $\F \cup \Ch$ is a chain, and that $\F^\bot \subset \Ch$.  Since $M , M^\bot \in \Ch$, indeed $\F \cup \F^\bot \cup \{ M , M^\bot \}$ is a chain.

We will show that $\F$ is a maximal closed generalized flag.  By the proof of Theorem \ref{main}, $\F$ is a bivalent closed generalized flag, so it remains to show that every good pair of $\F$ has codimension $1$.

Suppose, for the sake of a contradiction, that there exists a good pair $F' \subset F''$ of $\F$ with $\dim(F''/F') = \infty$.  Let $\Dh$ be a maximal chain of $\b$-stable subspaces between $F'$ and $F''$, and let $\G := \fl(\Dh)$.  Consider $\G = \{ G'_\beta, G''_\beta \}_{\beta}$.  It was seen in the proof of Theorem \ref{main} that $\overline{G'_\beta} = F''$ for all $\beta$.  That is, $(G'_\beta)^\bot = (F'')^\bot$ for all $\beta$.

Of course $\b$ stabilizes the generalized flag $0 \subset F' \subset \G \subset F'' \subset V$.  Now consider $\g \subset \gl(V, V_*)$, where in the isotropic cases $V_* = V$.  The stabilizer in $\gl(V,V_*)$ of the generalized flag $0 \subset F' \subset \G \subset F'' \subset V$ is 
\begin{eqnarray*}
F'  \otimes V_* + \sum_\beta G''_\beta \otimes (G'_\beta)^\bot + V \otimes (F'')^\bot & =  & F' \otimes V_* + \sum_\beta G''_\beta \otimes (F'')^\bot + V \otimes (F'')^\bot \\
& =  & F' \otimes V_* + V \otimes (F'')^\bot.
\end{eqnarray*}
Hence $\b \cdot F'' \subset F'$, i.e. $\b$ stabilizes any subspace between $F'$ and $F''$.  This contradicts the fact that there are no closed $\b$-stable subspaces between $F'$ and $F''$.

This concludes the proof that $\F$ is a maximal closed generalized flag in $V$.  It was previously noted that if  $\g$ is $\so(V)$ or $\sp(V)$, then $\F \cup \F^\bot \cup \{M, M^\bot \}$ is a chain.  Moreover, the proof of Theorem \ref{main} gives that $\b = \St_\F$.
\end{proof}

%%%%%%%%%%
%%%%%%%%%%
\section{Borel subalgebras}\label{borels}

%%%%%%
%%%%%%
\subsection{Borel subalgebras of $\sl_\infty$}\label{slsection}

In this section it is shown that Borel subalgebras of $\sl_\infty$ correspond to maximal closed generalized flags in the standard representation.  Let $\b \subset \sl(V,V_*)$ be a Borel subalgebra.  Here we denote by $\St_\F$ the stabilizer in $\sl(V,V_*)$ of any generalized flag $\F$ in $V$.

\begin{lemma} \label{slcyclic}
Let $\F$  be a maximal closed generalized flag in $V$.  For any $u \in V$, 
$$\St_\F \cdot u = 
\begin{cases}
F'_u & \overline{F'_u} = F''_u ;\\
F'_u & F'_u \subset F''_u \textrm{ is the only good pair of } \F; \\
F''_u & \textrm{otherwise}.
\end{cases}$$
\end{lemma}

\begin{proof}
Fix $u \in V$.  Consider $\F = \{ F'_\alpha, F''_\alpha \}_{\alpha \in A}$.  There are three cases to consider.

Suppose first that there exists $\alpha \in A$ for which $(F'_\alpha)^\bot \cap u^\bot \nsubseteq (F''_\alpha)^\bot$.  Then there exists $y \in (F'_\alpha)^\bot \cap u^\bot$ such that $y \notin (F''_\alpha)^\bot$.  Hence there exists $x \in F''_\alpha$ such that $\langle x , y \rangle = 1$.  Then $\St_\F = \Span_{\alpha \in A} \{ v \otimes w - \langle v , w \rangle x \otimes y : v \in F''_\alpha, w \in (F'_\alpha)^\bot \}$.   Let $v \in F''_\alpha$ and $w \in (F'_\alpha)^\bot$.  Since $(v \otimes w - \langle v , w \rangle x \otimes y )\cdot u = \langle u , w \rangle v - \langle u , y \rangle x = \langle u , w \rangle v$, indeed $\St_\F \cdot u = \bigcup_{u \notin \overline{F'_\alpha}} F''_\alpha$.  It is easy to check that $$\St_\F \cdot u = 
\begin{cases}
F'_u & \overline{F'_u} = F''_u ;\\
F''_u & \textrm{otherwise}.
\end{cases}$$

Suppose second that $(F'_\alpha)^\bot = (F''_\alpha)^\bot$ for all $\alpha \in A$.  Then $\St_\F = \sum_\alpha F''_\alpha \otimes (F'_\alpha)^\bot$, which is to say that the stabilizer of $\F$ in $\gl(V,V_*)$ is already traceless.  Let $v \in F''_\alpha$ and $w \in (F'_\alpha)^\bot$.  Since $(v \otimes w)\cdot u = \langle u , w \rangle v$, indeed $\St_\F \cdot u = \bigcup_{u \notin \overline{F'_\alpha}} F''_\alpha$.  Again $$\St_\F \cdot u = 
\begin{cases}
F'_u & \overline{F'_u} = F''_u ;\\
F''_u & \textrm{otherwise}.
\end{cases}$$

Suppose third that $(F'_\alpha)^\bot \cap u^\bot \subset (F''_\alpha)^\bot$ for all $\alpha \in A$ and that there exists $\gamma \in A$ for which $(F'_\gamma)^\bot \neq (F''_\gamma)^\bot$.  Then $(F'_\gamma)^\bot \cap u^\bot \subset (F''_\gamma)^\bot \subsetneq (F'_\gamma)^\bot$ implies that $(F'_\gamma)^\bot \cap u^\bot = (F''_\gamma)^\bot$.  Thus $(F''_\gamma)^\bot \subset u^\bot$, and hence $u \in \overline{F''_\gamma} = F''_\gamma$.  If $u \in F'_\gamma$, then $u^\bot \cap (F'_\gamma)^\bot = (F'_\gamma)^\bot$.  %Also, $(F'_\gamma)^\bot$ is not contained in $u^\bot$, hence $u \notin F'_\gamma$.  
Hence $u \in F''_\gamma \setminus F'_\gamma$.  This argument implies that $\F$ has exactly one good pair.  One may check that $\St_\F = (\sum_{\alpha \in A} F''_\alpha \otimes (F'_\alpha)^\bot) \cap \sl(V,V_*) = \sum_{\gamma \neq \alpha \in A} F''_\alpha \otimes (F'_\alpha)^\bot + F''_\gamma \otimes (F''_\gamma)^\bot$.  In this case, $\St_\F \cdot u = F'_u$.
\end{proof}

\begin{lemma} \label{slinjective}
If $\F$ and $\G$ are maximal closed generalized flags in $V$ with $\St_\F \subset \St_\G$, then $\F = \G$.
\end{lemma}

\begin{proof}
Let $\F = \{ F'_\alpha , F''_\alpha \}_{\alpha \in A}$ and $\G = \{G'_\beta, G''_\beta \}_{\beta \in B}$. For each $\alpha \in A$ choose $u_\alpha \in F''_\alpha \setminus F'_\alpha$.  If there is exactly one $\gamma \in A$ such that $\overline{G'_\gamma} = G'_\gamma$, define $A' := A\setminus \{\gamma \}$. Otherwise, let $A':=A$.

Since $\St_\F \subset \St_\G$, it follows that $$\overline{\St_\F \cdot u_\alpha} \subset \overline{\St_\G \cdot u_\alpha}.$$  For any $\alpha \in A'$, Lemma \ref{slcyclic} implies that $\overline{\St_\F \cdot u_\alpha} = F''_\alpha$.  Therefore for any $\alpha \in A'$, indeed $F''_\alpha = F''_{u_\alpha} = \overline{\St_\F \cdot u_\alpha} \subset \overline{\St_\G \cdot u_\alpha} \subset G''_{u_\alpha}$. 

We will show that $F''_\alpha = G''_{u_\alpha}$ for all $\alpha \in A'$.  There are two cases to consider.

\begin{enumerate}
\item  Suppose $\overline{F'_\alpha} = F''_\alpha$.  Then Lemma \ref{slcyclic} implies that for any $u \notin F''_\alpha$, indeed $F''_\alpha \subset \St_\F \cdot u$.  Observe that $G'_{u_\alpha}$ is stable under $\St_\F$ and $u_\alpha \notin G'_{u_\alpha}$.  It follows that $G'_{u_\alpha} \subset F''_\alpha$.  Thus $G'_{u_\alpha} \subset F''_\alpha \subset G''_{u_\alpha}$.  Since $u_\alpha \in F''_\alpha$, it must be that $G'_{u_\alpha} \subsetneq F''_\alpha$.  Since $\G$ is a maximal closed generalized flag, necessarily $F''_\alpha = G''_{u_\alpha}$.

\item  Suppose $F'_\alpha$ is closed.   Then Lemma \ref{slcyclic} implies that for any $u \notin F'_\alpha$, indeed $F''_\alpha \subset \St_\F \cdot u$.  Observe that $G'_{u_\alpha}$ is stable under $\St_\F$ and $u_\alpha \notin G'_{u_\alpha}$.  It follows that $G'_{u_\alpha} \subset F'_\alpha$.  Thus $G'_{u_\alpha} \subset F'_\alpha \subset F''_\alpha \subset G''_{u_\alpha}$.  Since $\G$ is a maximal closed generalized flag, necessarily $F''_\alpha = G''_{u_\alpha}$.
\end{enumerate}

If $A = A'$, then the proof is done, since a generalized flag is determined by its set of successors.  Assume therefore that $A \neq A'$, in which case it remains to show that $F''_\gamma = G''_{u_\gamma}$.  Observe first that $F'_\gamma =  \bigcup_{\alpha < \gamma} F''_\alpha = \bigcup_{\alpha < \gamma} G''_{u_\alpha}$.  Since $\dim F'_\alpha / F''_\gamma = \infty$ for any $\alpha > \gamma$, it must hold that $F''_\gamma = \bigcap_{\alpha > \gamma} F'_\alpha = \bigcap_{\alpha > \gamma} F''_\alpha = \bigcap_{\alpha > \gamma} G''_{u_\alpha}$.  Since $\dim F''_\gamma / F'_\gamma = 1$, it follows that $F'_\gamma \subset F''_\gamma$ is a pair in $\G$.  Thus $F''_\gamma = G''_{u_\gamma}$, and consequently $\F = \G$.
\end{proof}

The following result fully describes Borel subalgebras of $\sl(V,V_*)$.

\begin{thm} \label{slmain}
A subalgebra of $\sl(V,V_*)$ is a Borel subalgebra if and only if it is the stabilizer of a maximal closed generalized flag in $V$.  
Furthermore, the map $\F \mapsto \St_\F$ is a bijection between maximal closed generalized flags in $V$ and Borel subalgebras of $\sl(V,V_*)$.
\end{thm}

\begin{proof}
Let $\F$ be an arbitrary maximal closed generalized flag in $V$.  Because $\St_\F$ equals the stabilizer of any maximal generalized flag refining $\F$, Proposition \ref{locsolvable} yields that $\St_\F$ is locally solvable.  Hence there exists a Borel subalgebra $\b$ with $\St_\F \subset \b$.  By Proposition \ref{special}, there is a maximal closed generalized flag $\G$ in $V$ with $\b = \St_\G$.  It follows from Lemma \ref{slinjective} that $\F = \G$.  As a result, $\St_\F = \b$ is a Borel subalgebra.  Hence $\F \mapsto \St_\F$ gives a map from maximal closed generalized flags in $V$ to Borel subalgebras of $\sl(V,V_*)$.  Proposition \ref{special} implies that the map is surjective, and Lemma \ref{slinjective} implies that it is injective.
\end{proof}

%%%%%%
%%%%%%
\subsection{Borel subalgebras of $\so_\infty$} \label{sosection}

In this section it is shown that Borel subalgebras of $\so_\infty$ almost correspond to maximal closed isotropic generalized flags in the standard representation.  Let $\b \subset \so(V)$ be a Borel subalgebra.  Here we denote by $\St_\F$ the stabilizer in $\so(V)$ of any generalized flag $\F$ in $V$, and we denote by $\St_{\F,\gl}$ the stabilizer of $\F$ in $\gl(V,V)$.  Of course, $\St_\F = \St_{\F,\gl} \cap \so(V)$.  %If $X$ and $Y$ are subspaces of $V$, we denote their antisymmetrizer by $X \wedge Y := \{ x \otimes y - y \otimes x : x \in Y , y \in Y \} \subset \bigwedge^2 V$.

%\begin{lemma} \label{branching}  Suppose $L \subset V$ is a closed isotropic subspace with $\dim L^\bot / L =2$.  Then there are precisely two maximal isotropic subspaces containing $L$.
%\end{lemma}

%\begin{proof} 
%There exist $x , y \in L^\bot \setminus L$ such that $\langle x , x \rangle = \langle y , y \rangle = 0$ and $\langle x , y \rangle = 1$.  So $L^\bot = L \oplus \Span \{ x, y \}$.  Any maximal isotropic subspace containing $L$ is of the form $L \oplus \C(ax + by)$, for some $a,b \in \C$.  Since $\langle ax + by , ax + by \rangle = 2ab$, the only maximal isotropic subspaces containing $L$ are  $L \oplus \C x$ and $L \oplus \C y$.
%\end{proof}

\begin{defn}
Let $\F = \{ F'_\alpha , F''_\alpha \}_{\alpha \in A}$ and $\G = \{ G'_\beta , G''_\beta \}_{\beta \in B}$ be maximal closed isotropic generalized flags.  We say that $\F$ and $\G$ are \emph{twins} if $A$ and $B$ have maximal elements, denoted $\infty$, such that:
\begin{enumerate}
\item \label{C1} $\{ F'_\alpha , F''_\alpha \}_{\alpha \in A \setminus \{\infty\}} =  \{ G'_\beta , G''_\beta \}_{\beta \in B \setminus \{ \infty \}}$;
\item \label{C2} $F'_\infty$ is closed and $\dim (F'_\infty)^\bot / F'_\infty = 2$; and
\item \label{C3} $F''_\infty \neq G''_\infty$ are the two maximal isotropic subspaces containing $F'_\infty$.
\end{enumerate}
\end{defn}

Condition (\ref{C1}) of this definition forces $F'_\infty = G'_\infty$.  As for condition (\ref{C3}), it makes sense after condition (\ref{C2}) because whenever $L \subset V$ is a closed isotropic subspace with $\dim L^\bot / L = 2$, there are exactly two maximal isotropic subspaces containing $L$.  We say that $\F$ \emph{has a twin} if $\F =  \{ F'_\alpha , F''_\alpha \}_{\alpha \in A}$ is a maximal closed isotropic generalized flag with a maximal element $\infty$, such that $F'_\infty$ is closed and $(F''_\infty)^\bot = F''_\infty$.  If $\F$ has a twin, let $\tw(\F)$ denote the twin of $\F$.  That is, $\tw(\F)$ is obtained from $\F$ by replacing $F''_\infty$ with the other maximal isotropic subspace containing $F'_\infty$.  Note that $\tw$ is an involution on the set of maximal closed isotropic generalized flags that have twins.  Generalizing a phenomenon already present in the case of $\so_{2n}$, the maximal closed isotropic generalized flags $\F$ and $\tw(\F)$ have the same stabilizer in $\so(V)$. 

\begin{lemma} \label{sosamestabilizer}
Let $\F$ be a maximal generalized flag in $V$ such that $\F \cup \F^\bot \cup \{M , M^\bot \}$ is a chain for some maximal isotropic subspace $M \subset V$.  Then $\St_{\F_{iso}} = \St_\F$.
\end{lemma}

\begin{proof}
Clearly $\St_\F \subset \St_{\F_{iso}} $.  Let $Z \in \St_{\F_{iso}}$ be arbitrary.  

Let $\F = \{ F'_\alpha, F''_\alpha \}_{\alpha \in A}$.  First one must show that $\St_\F = \sum_{\alpha \in A, F''_\alpha \subset M} F''_\alpha \wedge (F'_\alpha)^\bot$.

For any $x \in F''_\alpha$ and $y \in (F'_\alpha)^\bot$, on the one hand $x \otimes y \in \St_{\F,\gl}$, but on the other hand since $\F \cup \F^\bot$ is a chain, in fact $y \otimes x \in \St_{\F,\gl}$.  In detail, there exists $\beta \in A$ for which $y \in F''_\beta \setminus F'_\beta$.  Since $\F \cup \F^\bot$ is a chain, and $y \in (F'_\alpha)^\bot$ and $y \notin F'_\beta$, it follows that $F'_\beta \subsetneq (F'_\alpha)^\bot$.  Moreover since $\dim (F'_\alpha)^\bot / (F''_\alpha)^\bot \leq 1$, it must be that $F'_\beta \subset (F''_\alpha)^\bot$, and thus $F''_\alpha = (F''_\alpha)^{\bot \bot} \subset (F'_\beta)^\bot$.  So $x \in (F'_\beta)^\bot$, and hence $y \otimes x \in F''_\beta \otimes (F'_\beta)^\bot \subset \St_{\F , \gl}$.
Thus the map of vector spaces (which is \emph{not} a map of Lie algebras):
\begin{eqnarray*}
\varphi: \sum_{\alpha \in A} F''_\alpha \otimes (F'_\alpha)^\bot & \rightarrow & {\bigwedge}^2 V \\
x \otimes y & \mapsto & x \otimes y - y \otimes x
\end{eqnarray*}
in fact has its image in $\St_\F$.  As $\varphi |_{\St_\F} = 2 \cdot \mathrm{Id}$, indeed $\varphi$ maps surjectively onto $\St_\F$.  Because $\sum F''_\alpha \otimes (F'_\alpha)^\bot$ is spanned by elements of the form $x \otimes y$, with $x \in F''_\alpha$ and $y \in (F'_\alpha)^\bot$ for some $\alpha \in A$, likewise $\St_\F$ is spanned by elements of the form $x \otimes y - y \otimes x$, with $x \in F''_\alpha$ and $y \in (F'_\alpha)^\bot$ for some $\alpha \in A$.

Suppose $M \neq M^\bot$.  Observe that $M \subset M^\bot$ is a pair in the generalized flag $\F$.  In this case, $\St_\F$ is in fact spanned by elements of the form $x \otimes y - y \otimes x$, with $x \in F''_\alpha$ and $y \in (F'_\alpha)^\bot$ for $\alpha \in A$ which are not equal to the pair $M \subset M^\bot$.  To see this, consider that the term in $\St_\F$ corresponding to the pair $M \subset M^\bot$ is $M^\bot \otimes M^\bot$.  Let $m \in M^\bot \setminus M$.  Observe that $M^\bot \otimes M^\bot = M^\bot \otimes M + M \otimes M^\bot + \C(m \otimes m)$.  Now $M^\bot \otimes M + M \otimes M^\bot \subset \sum_{(M \subset M^\bot) \neq \alpha \in A} F''_\alpha \otimes (F'_\alpha)^\bot$ and $m \otimes m \in \Sym^2(V)$. Since $\sigma$ fixes $M^\bot \otimes M + M \otimes M^\bot$ and $m \otimes m$, it follows that $\St_\F =  \big( \sum_{(M \subset M^\bot) \neq \alpha \in A} F''_\alpha \otimes (F'_\alpha)^\bot \big) \cap {\bigwedge}^2 V$.

Moreover, $\St_\F$ is spanned by elements of the form $x \otimes y - y \otimes x$, with $x \in F''_\alpha \subset M$ and $y \in (F'_\alpha)^\bot$.  (Explicitly, if $M^\bot \subset F'_\alpha$, then $y \in (F'_\alpha)^\bot \subset M^{\bot \bot} = M$.)  This concludes the proof that $\St_\F = \sum_{\alpha \in A, F''_\alpha \subset M} F''_\alpha \wedge (F'_\alpha)^\bot$.

Now $$\St_{\F_{iso} , \gl} =  \sum_{F''_\alpha \subset M} F''_\alpha \otimes (F'_\alpha)^\bot + V \otimes M^\bot,$$ because it is the stabilizer of  $\F_{iso} \cup \{M \subset V \}$, which is a generalized flag in $V$.  Then $Z = X + Y$ for some $X \in \sum_{F''_\alpha \subset M} F''_\alpha \otimes (F'_\alpha)^\bot$ and $Y \in V \otimes M^\bot$.

Note that $Z = -\sigma(Z)$, i.e. $X + Y = -\sigma(X) - \sigma(Y)$.  That is, $Y + \sigma(X) = -\sigma(Y) - X$, and the left hand side of this equation is clearly an element of $V \otimes M^\bot$, while the righthand side is clearly an element of $M^\bot \otimes V$.  So $Y + \sigma(X) \in (V \otimes M^\bot) \cap (M^\bot \otimes V) = M^\bot \otimes M^\bot$.  Now $\sigma(Y + \sigma(X)) = \sigma(Y) + X = -(Y + \sigma(X))$, and therefore $Y + \sigma(X) \in (M^\bot \otimes M^\bot) \cap \bigwedge^2(V) =  \bigwedge^2 M^\bot$.  Let $\eta := Y + \sigma(X) \in \bigwedge^2 M^\bot$.  So $Z = X - \sigma(X) + \eta$.  Clearly $X - \sigma(X) \in \St_\F$.  Observe that either $M \subset M^\bot$ is a pair in $\F$, in which case $M^\bot \otimes M^\bot \subset \St_{\F,\gl}$, or else $M = M^\bot$, in which case $M \otimes M \subset \St_{\F,\gl}$.  Since $\eta \in \bigwedge^2 M^\bot \subset \St_\F$, it follows that $Z \in \St_\F$.
\end{proof}

\begin{lemma} \label{soiso}
Let $\F = \{ F'_\alpha , F''_\alpha \}_{\alpha \in A}$ be a maximal closed isotropic generalized flag in $V$.  Then $\St_\F = \sum_{\alpha \in A} F''_\alpha \wedge (F'_\alpha)^\bot$, and moreover $\St_\F$ is locally solvable.
\end{lemma}

\begin{proof}
Let $M$ denote $\bigcup_{\alpha \in A} F''_\alpha$, which is a maximal isotropic subspace of $V$.  Let $\Ch$ be any maximal chain in $V$ containing $\F \cup \F^\bot$, and let $\H := \fl(\Ch)$.  Note that $\H_{iso}$ is a refinement of $\F$.  We will show that $\H \cup \H^\bot \cup \{M , M^\bot \}$ is a chain.  Then Lemma \ref{sosamestabilizer} gives that $\St_\H = \St_{\H_{iso}}$.  Since $\H_{iso}$ is a refinement of $\F$, and moreover since $\F$ is a maximal closed isotropic generalized flag in $M$, it holds that $\St_{\H_{iso}} = \St_\F$.  By Proposition \ref{locsolvable}, $\St_{\H, \gl}$ is locally solvable.  Hence $\St_\F = \St_{\H_{iso}} = \St_\H \subset \St_{\H, \gl}$ is locally solvable.  The formula for $\St_\F$ is seen in the proof of Lemma \ref{sosamestabilizer} to be a formula for $\St_\H = \St_{\H_{iso}}$.

It remains to show that $\H \cup \H^\bot \cup \{M , M^\bot \}$ is a chain.  Clearly $M$ and $M^\bot = \bigcap_{F \in \F} F^\bot$ are automatically compatible with $\Ch$, and they remain compatible with $\H$, and consequently also with $\H^\bot$.  Now suppose $H, I \in \H$, and one must show that either $H^\bot \subset I$ or $I \subset H^\bot$.  If $H$ and $I$ are both isotropic, then $I \subset M \subset M^\bot \subset H^\bot$.  If $H$ and $I$ are both coisotropic, then $H^\bot \subset M \subset M^\bot \subset I$.  It remains to deal with the cases
\begin{itemize}
\item $H \subset M \subset M^\bot \subset I$;
\item $I \subset M \subset M^\bot \subset H$.
\end{itemize}

In the first case, $F' \subset H \subset F''$ for some immediate predecessor-successor pair $F' \subset F''$ in $\F$.  Thus $(F'')^\bot \subset H^\bot \subset (F')^\bot$.  Since $\dim (F')^\bot / (F'')^\bot \leq 1$, it must be the case that $H^\bot$ is either $(F')^\bot$ or $(F'')^\bot$.  Since $\F^\bot \cup \{I\}$ is a chain, $H^\bot$ either contains or is contained in $I$.

In the second case, $F' \subset I \subset F''$ for some immediate predecessor-successor pair $F' \subset F''$ in $\F$.  Since $\dim (F')^\bot / (F'')^\bot \leq 1$, either $H \subset (F'')^\bot$ or $(F')^\bot \subset H$.  If $H \subset (F'')^\bot$, then $I \subset F'' = \overline{F''} \subset H^\bot$, i.e. $I \subset H^\bot$.  

Now assume that $(F')^\bot \subset H$.  Suppose there exists $F \in \F$ with $F^\bot \subset H$ and $F \subsetneq F'$.  Then $H \subset \overline{F} \subset F' \subset I$, and there is nothing left to show.  It remains to treat the case when $H \subset F^\bot$ for all $F \in \F$ with $F \subsetneq F'$.  Of course $F' = \bigcup_{F \subsetneq F'} F$.  Hence $(F')^\bot = ( \bigcup_{F \subsetneq F'} F)^\bot =  \bigcap_{F \subsetneq F'} F^\bot$.  So $H \subset \bigcap_{F \subsetneq F'} F^\bot = (F')^\bot \subset H$.  Hence $H = (F')^\bot$, i.e. $H^\bot = \overline{F'}$ which is either $F'$ or $F''$.  Since $\F \cup \{I\}$ is a chain, $H^\bot$ either contains or is contained in $I$.  
\end{proof}

\begin{lemma} \label{socyclic}
Let $\F$ be a maximal closed isotropic generalized flag in $V$.  If $u \in \bigcup_{F \in \F} F$, then 
$$\St_\F \cdot u = 
\begin{cases}
F''_u & \overline{F'_u} = F'_u \\
F'_u & \overline{F'_u} = F''_u.
\end{cases}$$
Thus $\overline{\St_\F \cdot u} = F''_u$.
\end{lemma}

\begin{proof}
Let $\F = \{ F'_\alpha ,F''_\alpha \}_{\alpha \in A}$, and let $u \in  \bigcup_\alpha F''_\alpha$.  Lemma \ref{soiso} states that $\St_\F = \sum_\alpha F''_\alpha \wedge (F'_\alpha)^\bot$.  Fix $\beta \in A$, and let $x \in F''_\beta$ and $y \in (F'_\beta)^\bot$.  Then $(x \otimes y - y \otimes x)\cdot u = \langle u , y \rangle x - \langle u , x \rangle y = \langle u , y \rangle x$, since $\langle \bigcup_\alpha F''_\alpha, \bigcup_\alpha F''_\alpha \rangle = 0$.  Hence $\St_\F \cdot u = \bigcup_{u \notin \overline{F'_\beta}} F''_\beta$.  The lemma follows easily.
\end{proof}

\begin{lemma} \label{controlmax}
Suppose $\F = \{F'_\alpha , F''_\alpha \}_{\alpha \in A}$ and $\G =\{G'_\beta , G''_\beta \}_{\beta \in B}$ are maximal closed isotropic generalized flags in $V$ with $\St_\F \subset \St_\G$.  If $\bigcup_{\alpha \in A} F''_\alpha \neq \bigcup_{\beta \in B} G''_\beta$, then $A$ and $B$ have maximal elements $\infty$ with $F'_\infty = G'_\infty$ closed and $\dim (F'_\infty)^\bot / F'_\infty = 2$. 
\end{lemma}

\begin{proof}
Let $M := \bigcup_{F \in \F} F$ and $N := \bigcup_{G \in \G} G$, and suppose $M \neq N$. The maximality of $\F$ and $\G$ implies that both $M$ and $N$ are maximal isotropic subspaces.  Thus neither $M$ nor $N$ contains the other, and there exist $m \in M \setminus N$ and $n \in N \setminus M$.  There exists $\alpha \in A$ for which $m \in F''_\alpha \setminus F'_\alpha$.  For any $y \in (F'_\alpha)^\bot$, it holds that $m \otimes y - y \otimes m \in \St_\F \subset \St_\G$.  Since $m \notin N$, indeed $y - cm \in N$ for some $c \in \C$.  Hence $(F'_\alpha)^\bot \subset N \oplus \C m$.  As a result $M \subset M^\bot \subset (F'_\alpha)^\bot \subset N \oplus \C m$.  Hence $M = (M \cap N) \oplus \C m$.  Of course $M \cap N$ is necessarily closed, being the intersection of two closed subspaces of $V$. 

Consider the chain $M \cap N \subset M \subset M^\bot \subset (M \cap N)^\bot$.  Since $\dim M^\bot / M \leq 1$, and $ \dim M / (M \cap N) = \dim (M \cap N)^\bot / M^\bot =1$, it must be that $\dim (M \cap N)^\bot / (M \cap N)$ is either $2$ or $3$.

Also note that $n \in (M \cap N)^\bot$, since $N$ is isotropic.  Observe that $M \cap N \subset N \subset N^\bot \subset (M \cap N)^\bot$, and since $\dim N^\bot / N \leq 1$, and $\dim N / (M \cap N) = \dim (M \cap N)^\bot / N^\bot $, it must be the case that $\dim N / (M \cap N) = 1$.  Thus $N = (M \cap N) \oplus \C n$.

Suppose, for the sake of a contradiction, that $\dim (M \cap N)^\bot / (M \cap N) = 3$.  Then there exist $u \in M^\bot \setminus M$ and $v \in (M \cap N)^\bot \setminus M^\bot$ such that $\langle u , u \rangle = \langle m ,v \rangle = 1$, and $\langle u , v \rangle = \langle v , v \rangle = 0$.  So $n = am + bu + cv + x$ for some $a,b,c \in \C$ and $x \in M \cap N$.  Now $m \otimes u - u \otimes m \in \St_\F \subset \St_\G$, and $\St_\G \cdot N \subset N$, so 
\begin{eqnarray*}
(m \otimes u - u \otimes m ) \cdot n & = & (m \otimes u - u \otimes m) \cdot ( am + bu + cv + x) \\
& = & \langle  am + bu + cv + x , u \rangle m - \langle  am + bu + cv + x , m \rangle u \\
& = & bm - cu \\
& \in & N = (M \cap N) \oplus \C(am + bu + cv + x).
\end{eqnarray*}
Therefore $(bm - cu) - \lambda ( am + bu + cv) \in M \cap N$ for some $\lambda \in \C$.  So $(b - \lambda a)m - (c + \lambda b) u - \lambda c v = 0$, which implies $b = c = 0$.  It follows that $N = M$.  This contradicts the hypothesis that $M \neq N$.  It follows that $\dim (M \cap N)^\bot / (M \cap N) = 2$.  

Since $n \in (M \cap N)^\bot \setminus M^\bot$, necessarily $\langle m , n \rangle \neq 0$.  Assume without loss of generality that $\langle m , n  \rangle = 1$.  It remains to show that $A$ and $B$ have elements $\infty$ such that $F'_\infty = G'_\infty = M \cap N$.

We will show that if $F''_\alpha \nsubseteq M \cap N$, then $F''_\alpha = M$.  Suppose there exists $z \in F''_\alpha \setminus M \cap N$.  Rescaling $z$, it holds that $z = m + w$ for some $w \in M \cap N$.  For any $x \in M \cap N$, observe that $x \otimes n - n \otimes x \in \St_\F$, and $(x \otimes n - n \otimes x) \cdot z = (x \otimes n - n \otimes x) \cdot (m + w) = \langle m + w , n \rangle x - \langle m + w , x \rangle n = x $.  Hence $M \cap N \subset \St_\F \cdot F''_\alpha \subset F''_\alpha \subset M$.  Since $z \in F''_\alpha$ and $z \notin M \cap N$, it must be that $F''_\alpha =M$.

Likewise, if $G''_\beta \nsubseteq M \cap N$ then $G''_\beta = N$.  Suppose there exists $z \in G''_\beta \setminus M \cap N$.  The proof that $M \cap N \subset \St_\F \cdot G''_\beta$ is analogous to the argument in the above paragraph.  So $M \cap N \subset \St_\F \cdot G''_\beta \subset \St_\G \cdot G''_\beta \subset G''_\beta \subset N$.  Since $z \in G''_\beta$ and $z \notin M \cap N$, it must be that $G''_\beta = N$.

Thus each of $A$ and $B$ has a maximal element $\infty$ such that $F'_\infty = G'_\infty = M \cap N$.  It was already observed that $M \cap N$ is closed, and also that $\dim (M \cap N)^\bot / (M \cap N) = 2$. 
\end{proof}

\begin{lemma} \label{soinjective}
Suppose $\F$ and $\G$ are maximal closed isotropic generalized flags in $V$ with $\St_\F \subset \St_\G$.  If $\F \neq \G$, then $\F$ and $\G$ are twins; in either case, $\St_\F = \St_\G$.
\end{lemma}

\begin{proof}
Let $\F = \{ F''_\alpha , F'_\alpha \}_{\alpha \in A}$ and $\G = \{G''_\beta, G''_\beta \}_{\beta \in B}$.  If $\bigcup_{\alpha \in A} F''_\alpha \neq \bigcup_{\beta \in B} G''_\beta$, then let $\infty \in A , B$ be as in Lemma \ref{controlmax}, and define $A' := A \setminus \{ \infty \}$ and $B' := B \setminus \{ \infty \}$.  Otherwise, let $A':=A$ and $B' := B$.  For each $\alpha \in A'$ choose $u_\alpha \in F''_\alpha \setminus F'_\alpha$.  

Since $\St_\F \subset \St_\G$, of course $$\overline{\St_\F \cdot u_\alpha} \subset \overline{\St_\G \cdot u_\alpha}.$$  
Lemma \ref{controlmax} implies that $u_\alpha \in \bigcup_{\beta \in B'} G''_\beta$, so by Lemma \ref{socyclic} that is $F''_\alpha = F''_{u_\alpha} \subset G''_{u_\alpha}$. 

We will show that $F''_\alpha = G''_{u_\alpha}$ for all $\alpha \in A'$.  There are two cases to consider.

\begin{enumerate}
\item  Suppose $\overline{F'_\alpha} = F''_\alpha$.  Then Lemma \ref{slcyclic} implies that for any $u \notin F''_\alpha$, indeed $F''_\alpha \subset \St_\F \cdot u$.  Observe that $G'_{u_\alpha}$ is stable under $\St_\F$ and $u_\alpha \notin G'_{u_\alpha}$.  It follows that $G'_{u_\alpha} \subset F''_\alpha$.  Thus $G'_{u_\alpha} \subset F''_\alpha \subset G''_{u_\alpha}$.  Since $u_\alpha \in F''_\alpha$, it must be that $G'_{u_\alpha} \subsetneq F''_\alpha$.  Since $\G$ is a maximal closed generalized flag, necessarily $F''_\alpha = G''_{u_\alpha}$.

\item  Suppose $F'_\alpha$ is closed.   Then Lemma \ref{slcyclic} implies that for any $u \notin F'_\alpha$, indeed $F''_\alpha \subset \St_\F \cdot u$.  Observe that $G'_{u_\alpha}$ is stable under $\St_\F$ and $u_\alpha \notin G'_{u_\alpha}$.  It follows that $G'_{u_\alpha} \subset F'_\alpha$.  Thus $G'_{u_\alpha} \subset F'_\alpha \subset F''_\alpha \subset G''_{u_\alpha}$.  Since $\G$ is a maximal closed generalized flag, necessarily $F''_\alpha = G''_{u_\alpha}$.
\end{enumerate}

On the one hand, if $\bigcup_{\alpha \in A} F''_\alpha = \bigcup_{\beta \in B} G''_\beta$, then $\F = \G$, since a generalized flag is determined by its successors. If, on the other hand, $\bigcup_{\alpha \in A} F''_\alpha \neq \bigcup_{\beta \in B} G''_\beta$, then we have shown that $\{ F'_\alpha, F''_\alpha : \alpha \in A' \} = \{G'_\beta, G''_\beta : \beta \in B' \}$.  Lemma \ref{controlmax} implies that $F'_\infty = G'_\infty$ is a closed isotropic subspace with $\dim (F'_\infty)^\bot / F'_\infty = 2$.  There are precisely two maximal isotropic subspaces containing $F'_\infty$, and they must be $F''_\infty$ and $G''_\infty$, respectively.  Therefore $\G = \tw(\F)$.  We omit the proof of the fact that $\St_\F = \St_{\tw(\F)}$.
\end{proof}

The following result fully describes Borel subalgebras of $\so(V)$.

\begin{thm} \label{somain}
A subalgebra of $\so(V)$ is a Borel subalgebra if and only if it is the stabilizer of a maximal closed isotropic generalized flag in $V$.  Furthermore, a fiber of the map $$\F \mapsto \St_\F$$ from maximal closed isotropic generalized flags in $V$ to Borel subalgebras of $\so(V)$ is either a single maximal closed isotropic generalized flag which has no twin, or a pair of twins.
\end{thm}

\begin{proof}
Let $\b$ be a Borel subalgebra of $\so(V)$.  Proposition \ref{special} states that $\b$ is the stabilizer of a maximal closed generalized flag $\F$ in $V$ with $\F \cup \F^\bot \cup \{M , M^\bot \}$ being a chain for some maximal isotropic subspace $M \subset V$.  By Lemma \ref{sosamestabilizer}, $\b = \St_{\F_{iso}}$.  Observe that $\F_{iso}$ is a maximal closed isotropic generalized flag in $V$, since the union of the isotropic subspaces in $\F$ must be $M$.  Hence every Borel subalgebra of $\so(V)$ is the stabilizer of a maximal closed isotropic generalized flag in $V$.

Let $\F$ be an arbitrary maximal closed isotropic generalized flag in $V$.  By Lemma \ref{soiso}, $\St_\F$ is locally solvable.  Hence there exists a Borel subalgebra $\b$ with $\St_\F \subset \b$.  We have seen that there is a maximal closed isotropic generalized flag $\G$ with $\b = \St_\G$.  It follows from Lemma \ref{soinjective} that $\St_\F = \St_\G$.  This means that $\St_\F$ is a Borel subalgebra.  Hence $\F \mapsto \St_\F$ gives a map from maximal closed isotropic generalized flags in $V$ to Borel subalgebras of $\so(V)$.  Proposition \ref{special} implies that the map is surjective.  Lemma \ref{soinjective} implies that if $\St_\F = \St_\G$, then either $\F = \G$, or $\F$ and $\G$ are twins.  Since $\St_\F = \St_{\tw(\F)}$ whenever $\F$ has a twin, we have shown that a fiber of the map is either a single maximal closed isotropic generalized flag which has no twin, or a pair of twins.
\end{proof}

%%%%%%
%%%%%%
\subsection{Borel subalgebras of $\sp_\infty$} \label{spsection}

In this section it is shown that Borel subalgebras of $\sp_\infty$ correspond to maximal closed isotropic generalized flags in the standard representation.  Let $\b \subset \sp(V)$ be a Borel subalgebra.  Here we denote by $\St_\F$ the stabilizer in $\sp(V)$ of any generalized flag $\F$ in $V$, and we denote by $\St_{\F,\gl}$ the stabilizer of $\F$ in $\gl(V,V)$.  Of course, $\St_\F = \St_{\F,\gl} \cap \sp(V)$.  If $X$ and $Y$ are subspaces of $V$, we denote their symmetrizer by $X \& Y := \{ x \otimes y + y \otimes x : x \in Y , y \in Y \} \subset \Sym^2(V)$.

\begin{lemma} \label{spsamestabilizer}
Let $\F$ be a maximal generalized flag in $V$ such that $\F \cup \F^\bot \cup \{ M \}$ is a chain for some maximal isotropic subspace $M \subset V$.  Then $\St_{\F_{iso}} = \St_\F$.
\end{lemma}

\begin{proof}
Clearly $\St_\F \subset \St_{\F_{iso}} $.  Let $Z \in \St_{\F_{iso}}$ be arbitrary.  

Let $\F = \{ F'_\alpha, F''_\alpha \}_{\alpha \in A}$.  We first show that $\St_\F = \sum_{\alpha \in A , F''_\alpha \subset M} F''_\alpha \& (F'_\alpha)^\bot$.

For any $x \in F''_\alpha$ and $y \in (F'_\alpha)^\bot$, we know on the one hand that $x \otimes y \in \St_{\F,\gl}$, but on the other hand from the fact that $\F \cup \F^\bot$ is a chain, we find that also $y \otimes x \in \St_{\F,\gl}$.  In detail, we have $y \in F''_\beta \setminus F'_\beta$ for some $\beta \in A$.  Since $\F \cup \F^\bot$ is a chain, and $y \in (F'_\alpha)^\bot$ and $y \notin F'_\beta$, we have $F'_\beta \subsetneq (F'_\alpha)^\bot$.  Moreover since $\dim (F'_\alpha)^\bot / (F''_\alpha)^\bot \leq 1$, we have $F'_\beta \subset (F''_\alpha)^\bot$, and thus $F''_\alpha \subset (F''_\alpha)^{\bot \bot} \subset (F'_\beta)^\bot$.  So $x \in (F'_\beta)^\bot$, and we see that $y \otimes x \in F''_\beta \otimes (F'_\beta)^\bot \subset \St_{\F , \gl}$.
Hence the map of vector spaces (which is \emph{not} a map of Lie algebras):
\begin{eqnarray*}
\varphi: \sum_{\alpha \in A} F''_\alpha \otimes (F'_\alpha)^\bot & \rightarrow & \Sym^2(V) \\
x \otimes y & \mapsto & x \otimes y + y \otimes x
\end{eqnarray*}
in fact has its image in $\b$.  From the fact that $\varphi |_\b = 2 \cdot \mathrm{Id}$, we find that $\varphi$ maps surjectively onto $\b$.  Since $\sum F''_\alpha \otimes (F'_\alpha)^\bot$ is spanned by elements of the form $x \otimes y$, with $x \in F''_\alpha$ and $y \in (F'_\alpha)^\bot$, we see that $\b$ is spanned by elements of the form $x \otimes y + y \otimes x$, with $x \in F''_\alpha$ and $y \in (F'_\alpha)^\bot$.

In fact, $\b$ is spanned by elements of the form $x \otimes y + y \otimes x$, with $x \in F''_\alpha \subset M$ and $y \in (F'_\alpha)^\bot$.  (Explicitly, if $M \subset F'_\alpha$, then $y \in (F'_\alpha)^\bot \subset M^{\bot \bot} = M$.) 

Now $$\St_{\F_{iso} , \gl} =  \sum_{F''_\alpha \subset M} F''_\alpha \otimes (F'_\alpha)^\bot + V \otimes M,$$ because it is the stabilizer of  $\F_{iso} \cup \{M \subset V \}$, which is a generalized flag in $V$.  Then $Z = X + Y$ for some $X \in \sum_{F''_\alpha \subset M} F''_\alpha \otimes (F'_\alpha)^\bot$ and $Y \in V \otimes M$.

Since $Z = \sigma(Z)$, we have $X + Y = \sigma(X) + \sigma(Y)$.  So $Y - \sigma(X) = \sigma(Y) - X$, and the left hand side of the equation is clearly an element of $V \otimes M$, while the righthand side is clearly an element of $M \otimes V$.  So $Y - \sigma(X) \in (V \otimes M) \cap (M \otimes V) = M \otimes M$.  Now $\sigma(Y - \sigma(X)) = \sigma(Y) - X = Y - \sigma(X)$, and therefore $Y - \sigma(X) \in (M \otimes M) \cap \Sym^2(V) = \Sym^2(M)$.  Let $\eta := Y - \sigma(X) \in \Sym^2(M)$.  So $Z = X + \sigma(X) + \eta$.  Clearly $X + \sigma(X) \in \St_\F$.  Since $\eta \in \Sym^2(M) \subset \St_\F$, we have $Z \in \St_\F$.
\end{proof}

Lemmas \ref{spiso}, \ref{spcyclic}, and \ref{spinjective} may be proved in the same manner as the analogous statements in Section \ref{sosection} with only straightforward modifications needed, so the proofs are omitted.

\begin{lemma} \label{spiso}
Let $\F = \{F'_\beta, F''_\beta \}_{\beta \in B}$ be a maximal closed isotropic generalized flag in $V$.  Then $\St_\F = \sum_\beta F''_\beta \& (F'_\beta)^\bot$, and moreover $\St_\F$ is locally solvable.
\end{lemma}

\begin{lemma} \label{spcyclic}
Let $\F$ be a maximal closed isotropic generalized flag in $V$.  If $u \in \bigcup_{F \in \F} F$, then 
$$\St_\F \cdot u = 
\begin{cases}
F''_u & \overline{F'_u} = F'_u \\
F'_u & \overline{F'_u} = F''_u.
\end{cases}$$
Thus $\overline{\St_\F \cdot u} = F''_u$.
\end{lemma}

\begin{prop} \label{samemax}
If $\F$ and $\G$ are maximal closed isotropic generalized flags in $V$ with $\St_\F \subset \St_\G$, then $\bigcup_{F \in \F} F = \bigcup_{G \in \G} G$.
\end{prop}

\begin{proof}
Let $M := \bigcup_{F \in \F} F$ and $N := \bigcup_{G \in \G} G$.  The maximality of $\F$ and $\G$ implies that both $M$ and $N$ are maximal isotropic subspaces.  We will show that $\langle M , N \rangle = 0$.  Suppose, for the sake of a contradiction, that there exist $m \in M$ and $n \in N$ such that $\langle n , m \rangle \neq 0$.  Then $(m \otimes m) \cdot n = \langle n , m \rangle m$.  Since $\Sym^2(M) \subset \St_\F$, we have shown that $m \in \St_\F \cdot N \subset \St_\G \cdot N \subset N$.  But $\langle n , m \rangle \neq 0$, which contradicts the fact that $N$ is isotropic.  Hence $\langle M , N \rangle = 0$, and $N \subset M^\bot  = M$.  By the maximality of $N$, we have $M = N$.
\end{proof}

\begin{lemma} \label{spinjective}
If $\F$ and $\G$ are maximal closed isotropic generalized flags in $V$ with $\St_\F \subset \St_\G$, then $\F = \G$.
\end{lemma}

The following result fully describes Borel subalgebras of $\sp(V)$.

\begin{thm} \label{spmain}
A subalgebra of $\sp(V)$ is a Borel subalgebra if and only if it is the stabilizer of a maximal closed isotropic generalized flag in $V$.  Futhermore, the map from maximal closed isotropic generalized flags of $V$ to Borel subalgebras of $\sp(V)$ $$\F \mapsto \St_\F$$ is bijective.
\end{thm}

\begin{proof}
Let $\b$ be a Borel subalgebra of $\sp(V)$.  Proposition \ref{special} states that $\b$ is the stabilizer of a maximal closed generalized flag $\F$ in $V$ with $\F \cup \F^\bot \cup \{M \}$ being a chain for some maximal isotropic subspace $M \subset V$.  By Lemma \ref{sosamestabilizer}, $\b = \St_{\F_{iso}}$.  Observe that $\F_{iso}$ is a maximal closed isotropic generalized flag in $V$, since the union of the isotropic subspaces in $\F$ must be $M$.  Hence every Borel subalgebra of $\sp(V)$ is the stabilizer of a maximal closed isotropic generalized flag in $V$.  

Now let $\F$ be an arbitrary maximal closed isotropic generalized flag in $V$.  By Lemma \ref{spiso}, $\St_\F$ is locally solvable.  Hence there exists a Borel subalgebra $\b$ with $\St_\F \subset \b$.  We have seen that there is a maximal closed isotropic generalized flag $\G$ with $\b = \St_\G$.  It follows from Lemma \ref{soinjective} that $\F = \G$.  As a result, $\St_\F = \b$ is a Borel subalgebra.  Hence $\F \mapsto \St_\F$ gives a map from maximal closed isotropic generalized flags in $V$ to Borel subalgebras of $\sp(V)$.  Proposition \ref{special} gives that the map is surjective, and Lemma \ref{spinjective} implies that it is injective.
\end{proof}

%%%%%%
%%%%%%
\subsection{Explicit formulas}\label{formulas}

One can find a nice kind of toral subalgebra in a Borel subalgebra $\b$ of one of the simple infinite-dimensional root-reductive Lie algebras.  In each case, there exist toral subalgebras $\t \subset \b$ such that $\b = \t + \n$, where $\n$ denotes the nilpotent subalgebra of $\b$.   Hence irreducible representations of $\b$ are given by characters of $\t$.  The relevant formulas are shown in Figure \ref{fig1}.  A similar analysis is seen in the case of $\gl_\infty$ in \cite{DP2}.  For more about toral subalgebras, see \cite{DPS}.

If $\b \subset \sl(V,V_*)$ is a Borel subalgebra, then $\b$ is the stabilizer in $\sl(V,V_*)$ of a unique maximal closed generalized flag $\F = \{ F'_\alpha, F''_\alpha \}_{\alpha \in A}$ in $V$.  It is also the stabilizer in $\sl(V,V_*)$ of a unique maximal closed generalized flag $\G = \{G'_\beta, G''_\beta \}_{\beta \in B}$ in $V_*$.  Let $C$ denote the good pairs of $A$, and we may also identify $C$ with the subset of good pairs of $B$.  There exist $1$-dimensional subspaces $L_\gamma \subset V$ and $M_\gamma \subset V_*$ for $\gamma \in C$ such that $\langle L_\gamma ,  M_c \rangle = \delta_{\gamma c} \C$, and such that $F''_\gamma = F'_\gamma \oplus L_\gamma$ and $(F'_\gamma)^\bot = (F''_\gamma)^\bot \oplus M_\gamma$.  In fact, one can go so far as to require that there exist vector space complements $X_\alpha$ of $F'_\alpha$ in $F''_\alpha$ for $\alpha \in A \setminus C$, and vector space complements $Y_\beta$ of $G'_\beta$ in $G''_\beta$ for $\beta \in B \setminus C$, such that $V = (\bigoplus_{\gamma \in C} L_\gamma) \oplus (\bigoplus_{\alpha \in A \setminus C} X_\alpha)$ and $V_* = (\bigoplus_{\gamma \in C} M_\gamma) \oplus (\bigoplus_{\beta \in B \setminus C} Y_\beta)$.  The associated toral subalgebra of $\b$ is $$\t = (\bigoplus_{\gamma \in C} L_\gamma \otimes M_\gamma) \cap \sl(V,V_*).$$  The same construction produces toral subalgebras inside of Borel subalgebras of $\gl(V,V_*)$ as well, and the formulas for $\gl(V,V_*)$ are also given in Figure \ref{fig1}.

If $\b \subset \so(V)$ is a Borel subalgebra, then $\b$ is the stabilizer in $\so(V)$ of a maximal closed isotropic generalized flag $\F = \{ F'_\alpha, F''_\alpha \}_{\alpha \in A}$ in $V$.  Let $\G := \fl(\F^\bot \cup \{V\}) = \{ G'_\beta , G''_\beta \}_{\beta \in B}$.  There exist $1$-dimensional subspaces $L_\gamma \subset V$ and $M_\gamma \subset V$ for $\gamma \in C$ such that $\langle M_\gamma,M_c \rangle = 0$ and $\langle L_\gamma , M_c \rangle = \delta_{\gamma c} \C$, and such that $F''_\gamma = F'_\gamma \oplus L_\gamma$ and $(F'_\gamma)^\bot = (F''_\gamma)^\bot \oplus L_{-\gamma}$ for all $\gamma \in C$.  In fact, one can go so far as to require that there exist vector space complements $X_\alpha$ of $F'_\alpha$ in $F''_\alpha$ for $\alpha \in A \setminus C$, and vector space complements $Y_\beta$ of $G'_\beta$ in $G''_\beta$ for $\beta \in B \setminus C$, as well as a vector space complement $S$ (necessarily of dimension $0$ or $1$) of $\bigcup_\alpha F''_\alpha$ in $(\bigcup_\alpha F''_\alpha)^\bot$, such that $V = (\bigoplus_{\gamma \in C} L_\gamma \oplus M_\gamma) \oplus (\bigoplus_{\alpha \in A \setminus C} X_\alpha) \oplus (\bigoplus_{\beta \in B \setminus C} Y_\beta) \oplus S$.  The associated toral subalgebra is $$\t = \bigoplus_{\gamma \in C} L_\gamma \wedge M_\gamma.$$

If $\b \subset \sp(V)$ is a Borel subalgebra, then $\b$ is the stabilizer in $\sp(V)$ of a unique maximal closed isotropic generalized flag $\F = \{ F'_\alpha, F''_\alpha \}_{\alpha \in A}$ in $V$.  Let $\G := \fl(\F^\bot \cup \{V\}) = \{ G'_\beta , G''_\beta \}_{\beta \in B}$.  Let $C$ denote the good pairs of $A$, and we may also consider $C$ as a subset of $B$.  There exist $1$-dimensional subspaces $L_\gamma \subset V$ and $M_\gamma \subset V$ for $\gamma \in C$ such that $\langle L_\gamma , M_c \rangle = \delta_{\gamma c} \C$ and $\langle M_\gamma , M_c \rangle = 0$, and such that $F''_\gamma = F'_\gamma \oplus L_\gamma$ and $(F'_\gamma)^\bot = (F''_\gamma)^\bot \oplus M_\gamma$.  In fact, one can go so far as to require that there exist vector space complements $X_\alpha$ of $F'_\alpha$ in $F''_\alpha$ for $\alpha \in A \setminus C$, and vector space complements $Y_\beta$ of $G'_\beta$ in $G''_\beta$ for $\beta \in B \setminus C$, such that $V = (\bigoplus_{\gamma \in C} L_\gamma \oplus M_\gamma) \oplus (\bigoplus_{\alpha \in A \setminus C} X_\alpha) \oplus (\bigoplus_{\beta \in B \setminus C} Y_\beta)$.  The associated toral subalgebra is $$\t = \bigoplus_{\gamma \in C} L_\gamma \& M_\gamma.$$ 

\begin{figure}
%\caption{Formulas for the Borel subalgebra $\b$ associated to a maximal closed (isotropic) generalized flag $\F = \{ F'_\alpha , F''_\alpha \}_{\alpha \in A}$ in $V$, as well as its nilpotent subalgebra $\n$ and the toral subalgebra associated to a choice of $1$-dimensional subspaces $L_\gamma$ and $M_\gamma$. }
\caption{Formulas for the stabilizer $\b \subset \g$ of a maximal closed (isotropic) generalized flag $\{ F'_\alpha , F''_\alpha \}$ in $V$, the nilpotent subalgebra $\n$ of $\b$, and the toral subalgebra $\t$ associated to lines $L_\gamma$ and $M_\gamma$. }
 \label{fig1}
$$\begin{array}{c|c|c|c}
\g & \b & \n & \t \\ \hline

\gl(V,V_*) &  \sum_\alpha F''_\alpha \otimes (F'_\alpha)^\bot &  \sum_\alpha F''_\alpha \otimes (F''_\alpha)^\bot & \bigoplus_{\gamma \in C} L_\gamma \otimes M_\gamma \\ \hline

\sl(V,V_*) &  (\sum_\alpha F''_\alpha \otimes (F'_\alpha)^\bot) \cap \sl(V,V_*) &  \sum_\alpha F''_\alpha \otimes (F''_\alpha)^\bot & (\bigoplus_{\gamma \in C} L_\gamma \otimes M_\gamma) \cap \sl(V,V_*) \\ \hline

\so(V) &  \sum_\alpha F''_\alpha \wedge (F'_\alpha)^\bot &\sum_\alpha F''_\alpha \wedge (F''_\alpha)^\bot &  \bigoplus_{\gamma \in C} L_\gamma \wedge M_\gamma \\ \hline

\sp(V) &  \sum_\alpha F''_\alpha \& (F'_\alpha)^\bot &\sum_\alpha F''_\alpha \& (F''_\alpha)^\bot &  \bigoplus_{\gamma \in C} L_\gamma \& M_\gamma 
\end{array}$$
\end{figure}

%%%%%%
%%%%%%
\subsection{Two examples}

Let $V$ be the vector space with basis $\{ x_i : i \in \Z_{\neq 0} \}$, and let $V_*$ be the span of the elements $\{ x_i^* \in V^* : i \in \Z_{\neq 0} \}$, where $\langle \cdot , \cdot \rangle : V \times V_* \rightarrow \C$ is defined by $\langle x_i , x_j^* \rangle := \delta_{ij}$.  
Consider for $i \in \Z_{\neq 0}$ the subspace $F_i := \Span \{ x_j : j \leq i \} \subset V$.  For each $i$ the subspace $F_i \subset V$ is closed.  The chain
$$ \cdots \subset F_{-2} \subset F_{-1} \subset F_1 \subset F_2 \subset \cdots$$
 is a maximal closed generalized flag in $V$, and let $\b$ denote its stabilizer in $\sl(V,V_*)$, which is a Borel subalgebra of $\sl(V,V_*)$.  This example arises naturally from the finite-dimensional picture, since $\b$ is the union of Borel subalgebras of finite-dimensional subalgebras isomorphic to $\sl_n$ exhausting $\sl(V,V_*)$.  Explicitly, let $V_n := \Span \{ x_j : -n \leq j \leq n \} \subset V$ and $(V_n)_* := \Span \{ x_j^* : -n \leq j \leq n \} \subset V_*$, and define $\g_n := sl(V,V_*) \cap V_n \otimes (V_n)_*$.  Then $\g_n \cong \sl_{2n}$, and one may check that $\b \cap \g_n$ is a Borel subalgebra of $\g_n$.

For the second example, let $\g$ be the Lie algebra $\sl(V,V_*) \subsetplus \C X$, where one takes $X$ to have the same commutation relations as the formal sum $\sum_{i > 0} x_i \otimes (x_i^* + x_{-i}^*)$, in the notation of the first example.  One may check that $\g$ is a root-reductive Lie algebra.  The Borel subalgebra $\b$ of the first example is a locally solvable subalgebra of $\g$.  In fact $\b$ is a Borel subalgebra of $\g$.  To check this claim, it suffices to show that $\b$ is self-normalizing in $\g$, in light of Proposition \ref{normalizer} below.  Suppose $Y \in \sl(V,V_*)$ and $a \in \C$ are such that $Y + a X \in \n_\g (\b)$.  Then $Y \in \g_n$ for some $n$.  Consider the element $Z := x_{n+1} \otimes x_{n+1}^* - x_{n+2} \otimes x_{n+2}^* \in \b$, and compute
\begin{eqnarray*}
[Y + aX , Z] & = &  a[X , Z ] \\
& = & a [x_{n+1} \otimes (x_{n+1}^* + x_{-n -1}^*) + x_{n+2} \otimes (x_{n+2}^* + x_{-n -2}^*), Z] \\
& = & a (-x_{n+1} \otimes x_{-n-1}^* + x_{n+2} \otimes x_{-n-2}^*).
\end{eqnarray*}
Since $-x_{n+1} \otimes x_{-n-1}^* + x_{n+2} \otimes x_{-n-2}^* \notin \b$, it must be that $a=0$.  Hence $\n_\g (\b) \subset \sl(V,V_*)$, so $\n_\g (\b) = \b$.  Thus $\b$ is an example of a Borel subalgebra of $\sl(V,V_*)$ which remains maximal locally solvable when considered as a subalgebra of $\sl(V,V_*) \subsetplus \C X$.

%%%%%%
%%%%%%
\subsection{General case} \label{general}

Theorem \ref{main} states that any Borel subalgebra of an infinite-dimensional indecomposable root-reductive Lie algebra is the simultaneous stabilizer of a Borel generalized flag in each of the standard representations.  That is, if $\g$ is an infinite-dimensional indecomposable root-reductive Lie algebra, the image of the map $\{ \F_m \} \mapsto \bigcap_m \St_{\F_m}$ from families of Borel generalized flags in the standard representations of $\g$ to subalgebras of $\g$ contains the Borel subalgebras of $\g$.  At the same time, the image of the map  $\{ \F_m \} \mapsto \bigcap_m \St_{\F_m}$ from families of maximal closed generalized flags in the standard representations of $\g$ to subalgebras of $\g$ is contained in the Borel subalgebras of $\g$.  It is not the case that the simultaneous stabilizer of any family of Borel generalized flags in the standard representations is a Borel subalgebra.  For instance, there exist Borel generalized flags in $V$ which are not maximal closed, and the stabilizer of any such flag is not a Borel subalgebra of $\sl(V,V_*)$.
  
We can calculate the intersection of a Borel subalgebra $\b$ of an infinite-dimensional indecomposable root-reductive Lie algebra $\g$ with any simple direct summand of $[\g,\g]$.  Let $[\g,\g] \cong \bigoplus_m \s_m$ be the decomposition into simple direct summands, and let $V_m$ be the standard representations of $\g$.  Using Theorem \ref{main}, we know that for each $m$ there exist a bivalent closed generalized flag $\F_m$ in $V_m$ and a Borel generalized flag $\G_m$ refining $\F_m$ such that $\b = \bigcap_m \St_{\G_m}$.  Fix $m$, and consider $\F_m = \{ F'_\alpha , F''_\alpha \}_{\alpha \in A}$.  Let $B \subset A$ denote the pairs $\alpha$ such that $F'_\alpha$ is closed and $\dim F''_\alpha / F'_\alpha = \infty$.  Then one may check via a calculation similar to one in the proof of Proposition \ref{special} that $\b \cap \s_m = \St_{\G_m} \cap \s_m = (\sum_{\alpha \in A \setminus B} F''_\alpha \otimes (F'_\alpha)^\bot + \sum_{\beta \in B} F''_\beta \otimes (F''_\beta)^\bot  ) \cap \s_m$.  Clearly if $B$ is nonempty, then $\b \cap \s_m$ is not a Borel subalgebra of $\s_m$.

\begin{prop}\label{normalizer}
If $\b \subset [\g,\g]$ is a Borel subalgebra of $[\g,\g]$, then the normalizer $\n_\g (\b)$ is a Borel subalgebra of $\g$.  Moreover, if $\b'$ is a Borel subalgebra of $\g$ containing $\b$, then $\b' = \n_\g (\b)$.
\end{prop}

\begin{proof}
Note that $\b \subset \n_\g (\b) \cap [\g,\g]$.  For any $X \in \n_\g (\b) \cap [\g,\g]$, it must be that $\b + \C X$ is a locally solvable subalgebra of $[\g,\g]$, since $[\b + \C X , \b + \C X] \subset \b$.  By the maximality of $\b$, we know $\b = \b + \C X$, i.e. $X \in \b$.  Therefore $\b = \n_\g (\b) \cap [\g,\g]$.

Compute  $[ \n_\g (\b) , \n_\g (\b) ] \subset \n_\g (\b) \cap [\g,\g] = \b$.  As a result $\n_\g (\b)$ is a locally solvable subalgebra of $\g$.  

Let $\b'$ be any Borel subalgebra of $\g$ containing $\b$.  Then $\b' \cap [\g,\g]$ is a locally solvable subalgebra of $[\g,\g]$ containing $\b$.  By the maximality of $\b$, it holds that $\b = \b' \cap [\g,\g]$.  Therefore $[\b',\b] \subset [\b',\b'] \subset \b' \cap [\g,\g] = \b$, i.e. $\b' \subset \n_\g (\b)$.  By the maximality of $\b'$, we have $\b' = \n_\g (\b)$.  Thus $\n_\g (\b)$ is a Borel subalgebra of $\g$, and the second statement of the proposition also follows. 
\end{proof}

As a corollary of Proposition \ref{normalizer}, the simultaneous stabilizer in $\g$ of a maximal closed (isotropic) generalized flag in each of the standard representations is independent of the choices made in defining the action of $\g$ on its standard representations.  Another easy conseqence is the following theorem.

\begin{thm}\label{Borelcorrespondence}
Let $\g$ be an arbitrary root-reductive Lie algebra.  The map $\b \mapsto \n_\g(\b)$ yields a bijection from the set of Borel subalgebras of $[\g,\g]$ to the set of Borel subalgebras of $\g$ whose intersection with $[\g,\g]$ is a Borel subalgebra of $[\g,\g]$.
\end{thm}

\begin{proof}
Proposition \ref{normalizer} implies that the map $\b \mapsto \n_\g (\b)$ from Borel subalgebras of $[\g,\g]$ to subalgebras of $\g$ lands inside the set of Borel subalgebras of $\g$.  It was also seen in the proof that if $\b$ is a Borel subalgebra of $[\g,\g]$, then $\n_\g (\b) \cap [\g,\g] = \b$.  That is, the composition of the first  map with the map from Borel subalgebras of $\g$ to subalgebras of $[\g,\g]$ given by intersecting, i.e. $\b \mapsto \b \cap [\g,\g]$, is the map $\b \mapsto \b$.  The image of the map $\b \mapsto \n_\g (\b)$ is precisely the set of Borel subalgebras of $\g$ which yield Borel subalgebras when intersected with $[\g,\g]$.
\end{proof}

This yields a large class of Borel subalgebras of $\g$ which are in bijection with the Borel subalgebras of $[\g,\g]$.  Since $[\g,\g]$ decomposes into a direct sum of simple root-reductive Lie algebras, Borel subalgebras of $[\g,\g]$ can be understood as direct sums of Borel subalgebras of the simple direct summands of $[\g,\g]$.  This is a good context in which to view the results of this paper on Borel subalgebras of $\sl_\infty$, $\so_\infty$, and $\sp_\infty$, the three infinite-dimensional simple root-reductive Lie algebras.

The question remains open whether there exists a root-reductive Lie algebra $\g$ containing a Borel subalgebra $\b \subset \g$ such that $\b \cap [\g,\g]$ is not a Borel subalgebra of $[\g,\g]$.  If one could show that no such examples exist, then Theorem \ref{Borelcorrespondence} would become a classification of the Borel subalgebras of root-reductive Lie algebras.  This outcome would be nice in a way, yet it seems to me unlikely.  I would conjecture that this phenomenon does occur.  Such Borel subalgebras might seem pathological, but I do not see any simple way to preclude their existence.  Indeed, the commutator subalgebra $[\g,\g]$ is not as large in $\g$ as one might think.  As an illustration, a root-reductive Lie algebra $\g$ and a maximal toral subalgebra $\t \subset \g$ are constructed in \cite{DPS} with $\t \cap [\g,\g] = 0$, a far cry from a maximal toral subalgebra of $[\g,\g]$.  In light of this, one might reasonably hope to construct explicitly an example of a Borel subalgebra $\b \subset \g$ such that $\b \cap [\g,\g]$ is not a Borel subalgebra of $[\g,\g]$.  This last remaining gap in a basic understanding of Borel subalgebras of root-reductive Lie algebras would be closed by either producing such an example or proving that none exists.


\begin{thebibliography}{}
%\bibitem[\textbf{BB}]{BahturinBenkart}Y.  Bahturin, G. Benkart,  Highest weight modules for locally finite Lie algebras, AMS/IP Stud. in Adv. Math. 4 (1997), 1--31.
%\bibitem[\textbf{BS}]{BahturinStrade} Y. Bahturin, H. Strade,  Some examples of locally finite simple Lie algebras, Arch. Math. (Basel) 65 (1995), 23--26.
%\bibitem[\textbf{B1}]{B1} A. Baranov, Diagonal locally finite Lie algebras and a version of Ado's theorem, J. Algebra 199 (1998), 1--39. 
%\bibitem[\textbf{B2}]{B2} A. Baranov, Simple diagonal locally finite Lie algebras, Proc. London Math. Soc. 77 (1998), 362--386.
%\bibitem[\textbf{BZ}]{BaranovZhilinski} A. Baranov, A. Zhilinski, Diagonal direct limits of simple Lie algebras, Comm. Algebra 27 (1998), 2249--2766. 
%\bibitem[\textbf{B}]{Bourbaki} N. Bourbaki, Groupes et Alg\`ebres de Lie, Hermann, Paris, 1975.
\bibitem[\textbf{DP1}]{DP1} I. Dimitrov, I. Penkov, Weight modules of direct limit Lie algebras, Int. Math Res. Not. 1999 (5) (1999) 223--249.
\bibitem[\textbf{DP2}]{DP2} I. Dimitrov, I. Penkov, Borel subalgebras of $\gl(\infty)$, Resenhas do Instituto de Matem\'{a}tica e Estat\'{i}stica da Universidade de S\~{a}o Paulo 6 (2004), No. 2/3, 153-163.
\bibitem[\textbf{DP3}]{DP3} I. Dimitrov, I. Penkov, Borel and Cartan subalgebras of $\gl(\infty)$.  Manuscript.
\bibitem[\textbf{DPS}]{DPS} E. Dan-Cohen, I. Penkov, N. Snyder, Cartan subalgebras of root-reductive Lie algebras, J. Algebra 308 (2007), 583--611. 
\bibitem[\textbf{M}]{Mackey} G. Mackey, On infinite dimensional linear spaces, Trans. AMS 1945, Vol. 57, No. 2, 155-207.
%\bibitem[\textbf{NP}]{N-P} K.-H. Neeb, I. Penkov, Cartan subalgebras of $\gl_\infty$, Canad. Math. Bull. 46(2003), 597-616.
\end{thebibliography}
\end{document}